\documentclass{amsart}
\usepackage{amssymb, amsthm, amsmath, amsfonts,amscd,bm,fancyhdr}
\usepackage{graphics}
\usepackage{hyperref}
\usepackage[all]{xy}
\usepackage{enumerate}
\usepackage[mathscr]{eucal}
\usepackage{bbm}
\usepackage{tikz}
\usetikzlibrary{decorations.pathreplacing,angles,quotes,knots}

\providecommand{\U}[1]{\protect\rule{.1in}{.1in}}
\def\theenumi{\arabic{enumi}}

\def\theenumii{\alph{enumii}}
\def\p@enumii{\theenumi.}

\def\theenumiii{\arabic{enumiii}}
\def\p@enumiii{(\theenumi)(\theenumii)}

\def\p@enumiv{\p@enumiii.\theenumiii}
\theoremstyle{plain}
\newtheorem{theorem}{Theorem}[section]
\newtheorem{lemma}[theorem]{Lemma}
\newtheorem{proposition}[theorem]{Proposition}

\newtheorem{corollary}[theorem]{Corollary}

\numberwithin{equation}{section}

\theoremstyle{definition}

\newtheorem{definition}[theorem]{Definition}
\newtheorem{example}[theorem]{Example}
\newtheorem{remark}[theorem]{Remark}

\newtheorem{thmab}{Theorem}

\newtheorem{corab}[thmab]{Corollary}

\renewenvironment{proof}[1][\proofname]{{\bfseries #1\\}}{\qed}

\DeclareMathOperator{\FI}{FI}
\DeclareMathOperator{\VI}{VI}

\DeclareMathOperator{\Hom}{Hom}

\DeclareMathOperator{\Sym}{Sym}

\DeclareMathOperator{\Conf}{Conf}
\DeclareMathOperator{\DConf}{DConf}

\newcommand{\Sn}{\mathfrak{S}}
\newcommand{\as}{\text{*}}

\newcommand{\C}{{\mathcal{C}}}
\newcommand{\Ob}{\mathcal{O}}

\newcommand{\Z}{{\mathbb{Z}}}
\newcommand{\N}{\mathbb{N}}
\newcommand{\R}{\mathbb{R}}
\newcommand{\Q}{\mathbb{Q}}
\newcommand{\F}{\mathbb{F}}

\newcommand{\K}{\mathcal{K}}
\newcommand{\Par}{\mathcal{P}}

\newcommand{\fiHom}{\mathscr{H}om}

\newcommand{\dt}{\bullet}
\newcommand{\arXiv}[1]{\href{http://arxiv.org/abs/#1}{\nolinkurl{arXiv:#1}}}
\newcommand{\arXivV}[2]{\href{http://arxiv.org/abs/#1}{\nolinkurl{arXiv:#1v#2}}}

\title{Families of nested graphs with compatible symmetric-group actions}

\author{Eric Ramos and Graham White}
\address{University of Michigan Department of Mathematics, 530 Church St., Ann Arbor, MI 48109}
\email{egramos@umich.edu}
\address{Indiana University - Bloomington Department of Mathematics, Rawles Hall, Bloomington, IN 47405}
\email{grrwhite@iu.edu}

\thanks{The first author was supported by NSF grant DMS-1704811.}

\begin{document}

\begin{abstract}
For fixed positive integers $n$ and $k$, the Kneser graph $KG_{n,k}$ has vertices labeled by $k$-element subsets of $\{1,2,\dots,n\}$ and edges between disjoint sets. Keeping $k$ fixed and allowing $n$ to grow, one obtains a family of nested graphs, each of which is acted on by a symmetric group in a way which is compatible with all of the other actions. In this paper, we provide a framework for studying families of this kind using the $\FI$-module theory of Church, Ellenberg, and Farb \cite{CEF}, and show that this theory has a variety of asymptotic consequences for such families of graphs. These consequences span a range of topics including enumeration, concerning counting occurrences of subgraphs, topology, concerning Hom-complexes and configuration spaces of the graphs, and algebra, concerning the changing behaviors in the graph spectra.
\end{abstract}

\keywords{FI-modules, Representation Stability, Graph Theory, Kneser Graphs}

\maketitle

\section{Introduction}

\subsection{Motivation}

Let $\FI$ denote the category whose objects are the finite sets $[n] := \{1,\ldots,n\}$, and whose morphisms are injections. In their seminal work, Church, Ellenberg, and Farb introduced the notion of an \emph{$\FI$-module} to formalize the connection between a large number of seemingly unrelated phenomena in topology and representation theory \cite{CEF}. Formally, an $\FI$-module is a functor from $\FI$ to the category of real vector spaces. Noting that the endomorphisms in $\FI$ are permutations, one may imagine an $\FI$-module as a series of representations of the symmetric groups $\Sn_n$, with $n$ increasing, which are compatible in some sense.\\

Recently, there has been a push in the literature to use the same philosophy underlying $\FI$-modules to study combinatorial objects. For instance, in his recent work \cite{Gad2} Gadish studies what he calls $\FI$-posets and $\FI$-arrangements. In this work, we will be mostly focused on \emph{$\FI$-graphs}, functors from $\FI$ to the category of graphs. For us, a \emph{graph} is a compact 1-dimensional simplicial complex. Given a graph $G$, we write $V(G)$ for the set of vertices of $G$ and $E(G)$ for the set of edges. Just as with the work of Gadish, we will discover that a relatively simple combinatorial condition on $\FI$-graphs will allow us to conclude a plethora of interesting structural properties of the graphs which comprise it.\\

Throughout this paper we will often denote $\FI$-graphs by $G_\dt$, and use $G_n$ as a short-hand for its evaluation on $[n]$. The \emph{transition maps} of $G_\dt$ are the graph morphisms induced by the morphisms of $\FI$. We say that an $\FI$-graph $G_\dt$ is \emph{vertex-stable of stable degree $\leq d$} if for all $n \geq d$, every vertex of $G_n$ appears in the image of some transition map. Some common examples of vertex-stable $\FI$-graphs include:\\
\begin{itemize}
\item the complete graphs $K_n$;
\item the Kneser graphs $KG_{n,r}$, for each fixed $r$. These are the graphs whose vertices are $r$-element subsets of $[n]$, and whose edges indicate disjointness;
\item the Johnson graphs $J_{n,r}$, for each fixed $r$. These are the graphs whose vertices are $r$-element subsets of $[n]$, and whose edges indicate that the intersection of the two subsets has size $r-1$.\\
\end{itemize}
Other examples of vertex-stable $\FI$-graphs are given at the end of Section \ref{def}. While it is straightforward to verify that the above examples are vertex-stable, one might also observe that they have a variety of other symmetries. The main structure theorem of vertex-stable $\FI$-graphs is that the condition of vertex stability automatically yields several other symmetries.\\
\newpage
\begin{thmab}\label{structuretheorem}
Let $G_\dt$ be a vertex-stable $\FI$-graph. Then for all $n \gg 0$:
\begin{enumerate}
\item the transition maps originating from $G_n$ are injective;
\item the transition maps originating from $G_n$ have induced images (see Definition \ref{induced});
\item every edge of $G_{n+1}$ is the image of some edge of $G_n$ under some transition map;
\item for any fixed $r \geq 1$ and any collection of vertices $\{v_1,\ldots,v_r\}$ of $G_{n+1}$, there exists a collection of $r$ vertices of $G_n$, $\{w_1,\ldots,w_r\}$ which map to $\{v_1,\ldots,v_r\}$ under some transition map.\\
\end{enumerate}
\end{thmab}

One should note two recurring themes in the above theorem. Firstly, many of the results in this work (indeed, many of the results in the theory of $\FI$-modules) are only true asymptotically. Secondly, while one can prove the existence of certain behaviors in general, it is usually quite difficult to make such existential statements effective (see Theorem \ref{inducedtheorem} for an instance where this is not the case). This is a consequence of the methods used to prove such statements. In this work, the main proof techniques which will be employed fall under what one might call a Noetherian method. Namely, we rephrase what needs to be proven in terms of finite generation of some associated module. We then prove that this module is a submodule of something which is easily seen to be finitely generated, and apply standard Noetherianity arguments to conclude that the original module was finitely generated. It is an interesting question to ask which, if any, of our results can be made effective through more combinatorial means.\\

Following the proof of Theorem \ref{structuretheorem}, we spend the majority of the body of the paper illustrating various applications. These applications come in three flavors: enumerative, topological, and algebraic.\\

\subsection{Enumerative applications}

We begin by asking the following question: Given a vertex-stable $\FI$-graph $G_\dt$, is it possible to count the occurrences of some fixed substructure in $G_n$, as a function of $n$?\\

Given graphs $G$ and $H$, we say that $H$ is a \emph{subgraph} of $G$ if there exists an injective morphism of graphs $H \hookrightarrow G$. One instance of the above question is whether we can count the number of subgraphs of $G_n$ which are isomorphic to $H$, as a function of $n$. We answer this question in the affirmative.\\

\begin{thmab}\label{enumthmab}
Let $G_\dt$ be a vertex-stable $\FI$-graph of stable degree $\leq d$, and let $H$ be a graph. Then there exists a polynomial $p_H(x) \in \Q[x]$ of degree $\leq |V(H)|\cdot d$ such that for all $n \gg 0$ the function
\[
n \mapsto \text{ the number of subgraphs of $G_n$ isomorphic to $H$ }
\]
agrees with $p_H(n)$.\\
\end{thmab}

\begin{remark}
For a fixed pair of graphs $G$ and $H$, the number of subgraphs of $G$ isomorphic to $H$ is not the number of graph injections from $H$ to $G$. Indeed, usually one is concerned with counting the number of such injections \emph{up to postcomposition with automorphisms of $H$}. Because $H$ is independent of $n$, the above theorem remains true regardless of how the counting problem is interpreted.\\
\end{remark}

To convince themselves of this theorem, one should consider the case of the complete graphs $K_n$. In this case, one can count the number of occurrences of $H$ by first choosing $|V(H)|$ vertices, and then counting the number of copies of $H$ in the induced $K_{|V(H)|}$ subgraph. We will see in Section \ref{def} that $\FI$-graphs are fairly diverse, and therefore one should not expect the general case to be quite this straightforward. However, the idea that one should begin by choosing $|V(H)|$ vertices of $G_n$ remains relevant. From this point one proceeds by applying the fourth part of Theorem \ref{structuretheorem}.\\

Another interesting enumerative consequence of vertex stability involves counting degrees of vertices. Recall that in a given graph $G$, the \emph{degree} of a vertex $v$ is the number of edges adjacent to $v$. We usually write $\Delta(G)$ for the maximum degree of a vertex in $G$, and $\delta(G)$ for the minimum degree.\\
\newpage
\begin{thmab}\label{degreethm}
Let $G_\dt$ be a vertex-stable $\FI$-graph of stable degree $\leq d$. Then the functions
\[
n \mapsto \Delta(G_n) \hspace{0.4cm} \text{and} \hspace{0.4cm} n \mapsto \delta(G_n)
\]
each agree with a polynomial of degree at most $d$ for all $n \gg 0$.\\
\end{thmab}

While Theorem \ref{degreethm} appears very similar to Theorem \ref{enumthmab}, there is one subtle difference. In the case of Theorem \ref{enumthmab}, one reduces to the case of $\FI$-modules by considering the family of symmetric group representations induced by the symmetric group action on copies of $H$ inside $G_n$. It is unclear, however, whether such an approach can work to prove Theorem \ref{degreethm}, as the maximum and minimum degrees of $G_n$ cannot in any obvious way be realized as the dimension of some symmetric group representation. The proof of Theorem \ref{degreethm} is therefore a bit more subtle, and can be considered more traditionally combinatorial than that of Theorem \ref{enumthmab}.\\

To conclude our enumerative applications, we consider the question of counting walks in $G_n$. Recall that for a fixed integer $r \geq 0$ and a graph $G$, a \emph{walk of length $r$ in $G$} is an $(r+1)$-tuple of vertices of $G$, $(v_0,\ldots,v_r)$, such that for all $0 \leq i \leq r-1$, $\{v_i,v_{i+1}\} \in E(G)$. We say that a walk $(v_0,\ldots,v_r)$ is \emph{closed} if $v_r = v_0$.\\

\begin{thmab}
Let $G_\dt$ be a vertex-stable $\FI$-graph of stable degree $\leq d$. Then the functions
\[
n \mapsto |\{\text{walks in $G_n$ of length $r$}\}| \hspace{0.4cm} \text{and} \hspace{0.4cm} n \mapsto |\{\text{closed walks in $G_n$ of length $r$}\}|
\]
each agree with a polynomial of degree at most $rd$ whenever $n \gg 0$.\\
\end{thmab}

\subsection{Topological applications}

In this paper we will be primarily concerned with two topological applications of the theory of vertex-stable $\FI$-graphs. Our major results will prove that certain natural topological spaces associated to vertex-stable $\FI$-graphs will be \emph{representation stable} in the sense of Church and Farb \cite{CF} (see Definition \ref{repstable}). The first of our applications is related to the so-called $\Hom$-complexes.\\

Let $T$ and $G$ be two graphs. A \textbf{multi-homomorphism} from $T$ to $G$ is a map of sets,
\[
\alpha:V(T) \rightarrow \Par(V(G))-\emptyset
\]
such that for all edges $\{x,y\} \in E(T)$, and all choices of $v \in \alpha(x)$ and $w \in \alpha(y)$, one has $\{v,w\} \in E(G)$. The \emph{$\Hom$-complex} of $T$ and $G$, denoted $\fiHom(T,G)$, the simplicial complex whose simplicies are multi-homomorphisms between $T$ and $G$ (See Definition \ref{homcomplexdef} for details). These complexes first rose to popularity through the work of Babson and Koslov \cite{BK,BK2}, which expanded upon famous work of Lov\'asz \cite{Lo}. For instance, it is shown in those works that the topological connectivity of the space $\fiHom(K_2,G)$ can be used to bound the \emph{chromatic number} of $G$.\\

\begin{thmab}
Let $G_\dt$ be a vertex-stable $\FI$-graph. Then for any graph $T$, the functor
\[
n \mapsto \fiHom(T,G_n)
\]
is representation stable (see Definition \ref{repstable}). In particular, if $i \geq 0$ is fixed, then the function
\[
n \mapsto \dim_\R(H_i(\fiHom(T,G_n);\R))
\]
eventually agrees with a polynomial of degree at most $|V(T)|\cdot d(i+1)$.\\
\end{thmab}

While this result might seem somewhat technical, it has one particularly notable consequence about counting graph homomorphisms into $\FI$-graphs.\\

\begin{corab}
Let $G_\dt$ denote a vertex-stable $\FI$-graph of stable degree at most $d$. Then for any graph $T$ the function,
\[
n \mapsto |\Hom(T,G_n)|
\]
agrees with a polynomial of degree at most $|V(T)|\cdot d$ for all $n \gg 0$.\\
\end{corab}

It is a well known fact that $n$-colorings of vertices of a graph $T$ are in bijection with $\Hom(T,K_n)$, where $K_n$ is the complete graph on $n$ vertices. The above theorem can therefore be thought of as an extension of the theorem which posits the existence of the \emph{chromatic polynomial}.\\

\begin{remark}
The idea of treating the chromatic polynomial as an ``$\FI$ phenomenon'' was conveyed to the first author by John Wiltshire-Gordon and Jordan Ellenberg. This observation was a large part of the motivation for the present work.\\
\end{remark}

Following our treatment of the $\Hom$-complex, we next turn our attention to configuration spaces of graphs. Given a topological space $X$, the \emph{$n$-stranded configuration space of $X$} is the topological space of $n$ distinct points on $X$,
\[
\Conf_n(X) := \{(x_1,\ldots,x_n) \in X^n \mid x_i \neq x_j, i \neq j\}.
\]
Configuration spaces are in many ways the prototypical topological application of $\FI$-module theory. In fact, one of the results which eventually inspired the study of $\FI$-modules was Church's proof that configuration spaces of manifolds are often representation stable \cite{Ch}. It is an unfortunately true, however, that if $G$ is any graph then the family of topological spaces $\{\Conf_n(G)\}_n$ cannot be representation stable. In fact, they are extremely unstable in this sense, exhibiting factorial growth in their Betti numbers (see the discussion following Theorem \ref{eulerchar}). In this paper we therefore adapt a different approach, recently used by L\"utgehetmann \cite{Lu}. We consider the spaces $\Conf_m(G_\dt)$, where $m$ is fixed and $G_\dt$ is a vertex-stable $\FI$-graph.\\

\begin{thmab}\label{repstabconf}
Let $G_\dt$ be a vertex-stable $\FI$-graph with stable degree at most $d$ whose transition maps are all injective and whose constituent graphs $G_n$ are all connected. Then for any $m \geq 1$ the functor
\[
n \mapsto \Conf_m(G_n)
\]
is representation stable (see Definition \ref{repstable}). In particular, if $i \geq 0$ is fixed, then the function
\[
n \mapsto \dim_\R(H_i(\Conf_m(G_n);\R))
\]
eventually agrees with a polynomial of degree at most $2dm$.\\
\end{thmab}

\begin{remark}
Theorem \ref{structuretheorem} implies that the transition maps of any vertex-stable $\FI$-graph are eventually injective. Because the content of the previous theorem is asymptotic, we may always replace our $\FI$-graph with a new $\FI$-graph whose transition maps are injective and agrees with our original graph for all $n \gg 0$. In particular, the assumption that the transition maps of our $\FI$-graph must be injective is not particularly restrictive.\\

The condition that $G_n$ be connected is also not necessary, although the eventual conclusion is a bit less clean if it is not assumed. The most general version of Theorem \ref{repstabconf} is proven as Theorem \ref{strongconf} below.\\
\end{remark}

This theorem was proven for a particular $\FI$-graph (see Example \ref{lutex}) by L\"utgehetmann \cite{Lu}, although he did not use this language. His approach in that work is very topological, and sharpens certain bounds that we discover in this work, although it is limited to that example. Our approach is much more combinatorial in nature, and has the benefit of proving the above theorem for all vertex-stable $\FI$-graphs.\\

\subsection{Algebraic applications}

Our final kind of application involves studying the spectrum of vertex-stable $\FI$-graphs. For any graph $G$, let $\R V(G)$ denote the real vector space with basis indexed by the vertices of $G$. Then there are many natural endomorphisms of $\R V(G)$ which are of interest in algebraic graph theory. Perhaps the most significant is the \emph{adjacency matrix} of $G$. This is the matrix $A_G$ defined on vertices $v \in V(G)$ by
\[
A_Gv = \sum_{\{w,v\} \in E(G)} w
\]
The adjacency matrix of any graph is a real symmetric matrix, and therefore its eigenvalues must be real. This justifies the hypotheses of the following theorem.\\

\begin{thmab}
Let $G_\dt$ be a vertex-stable $\FI$-graph, and let $A_n$ denote the adjacency matrix of $G_n$. We may write the distinct eigenvalues of $A_n$ as,
\[
\lambda_1(n) < \lambda_2(n) < \ldots < \lambda_{r(n)}(n),
\]
for some function $r(n)$. Then for all $n \gg 0$
\begin{enumerate}
\item $r(n)$ is constant. In particular, the number of distinct eigenvalues of $A_n$ is eventually constant;
\item for any $i$ the function
\[
n \mapsto \text{ the multiplicity of $\lambda_i(n)$}
\]
agrees with a polynomial.\\
\end{enumerate}
\end{thmab}

\begin{remark}
The proof of the above theorem will be appearing in upcoming work of the first author and David Speyer. It is included in this paper for completeness's sake. Hints towards the proof are given in Section \ref{algap}.\\
\end{remark}

Perhaps the simplest example one can call upon to illustrate this theorem is the complete graph. In this instance we see that the eigenvalues of $A_n$ are given by $-1$ and $n-1$ with multiplicities $n-1$ and $1$, respectively. We see immediately from this that the number of distinct eigenvalues of $A_n$ becomes constantly 2 beginning at $n = 2$, and the multiplicities of these eigenvalues are given by polynomials.\\

\subsection{Outline}

The overall structure of the present work is as follows. We begin by recalling necessary background. This ranges from graph theory (Section 2.1) to the configuration spaces of graphs (Section 2.2) to the theory of $\FI$-modules and representation stability (Section 2.3). Our hope is that this background will be sufficient so that readers from a large variety of fields can better follow the work in the body of the paper.\\

Following this, we turn our attention to the basic definitions and examples from the theory of $\FI$-graphs (Section 3.1). We then describe the phenomenon of vertex-stability and its major structural consequences (Section 3.2). This third section is then capped off by a more technical chapter which solves the question of when the transition maps of a vertex-stable $\FI$-graph must begin to have induced image (Section 3.3). The fourth section is dedicated to proving the applications detailed above, as well as various smaller consequences that one might find interesting.\\

To conclude the work, we consider generalization of the theory of $\FI$-graphs in two distinct directions. Firstly, we consider what would happen if instead of $\FI$, one considered functors from certain other categories into the category of graphs (Section 5.1). In particular, we argue that virtually everything described in the paper will have some analog for $\FI^m$-graphs and $\VI(q)$-graphs (see Definition \ref{othercat}). Secondly, we consider higher dimensional analogs of $\FI$-graphs. Namely, we spend a bit of time considering general $\FI$-simplicial-complexes and show that certain structural facts will continue to work in this context (Section 5.2).\\

\section*{Acknowledgements}
We would like to send our deepest thanks to Tom Church, David Speyer, and Andrew Snowden for their helpful suggestions. We also thank Jordan Ellenberg, and John Wiltshire-Gordon for their early inspiration. The first author was supported by NSF grant DMS-1704811.\\

\section{Background}

\subsection{Graph Theory}

For the purposes of this paper, we will only consider finite graphs with no multi-edges or self-loops. Note that graphs will be permitted to be disconnected.\\

\begin{definition}\label{induced}
A \textbf{graph} is a 1-dimensional simplicial complex. Given a graph $G$, we will write $V(G)$ to denote its vertex set, and $E(G)$ to denote its edge set. If $v \in V(G)$, then $\mu(v)$ will be used to denote its \textbf{valency}, or \textbf{degree}. The \textbf{minimum degree} of a vertex of $G$ will be denoted $\delta(G)$, while the \textbf{maximum degree} will be written $\Delta(G)$.\\

A \textbf{homomorphism} of graphs $\phi:G \rightarrow G'$ is a map of sets $\phi:V(G) \rightarrow V(G')$ such that if $\{x,y\} \in E(G)$, then $\{\phi(x),\phi(y)\} \in E(G')$. The category of graphs and graph homomorphisms will be denoted $\mathbf{Graph}$.\\

A \textbf{subgraph} of a graph $G$ is a graph $G'$ with inclusions $V(G') \subseteq V(G)$ and $E(G') \subseteq E(G)$. We say that a subgraph $G'$ is \textbf{induced} if for all $x,y \in V(G')$, $\{x,y\} \in E(G')$ whenever $\{x,y\} \in E(G)$.
\end{definition}

In this work, we will be applying the theory of $\FI$-modules to the study of certain natural families of graphs. Our applications will be grouped into three categories: enumerative, topological, and algebraic.\\

To begin, we review some elementary facts and notations from enumerative graph theory. Much of what follows can be found in any standard text in graph theory (see, for instance, \cite{B}).\\

\begin{definition}
Let $G$ and $H$ be graphs. We write $\eta_H(G)$ to denote the total number of distinct subgraphs of $G$ which are isomorphic to $H$. We will also write $\eta^{ind}_H(G)$ to denote the total number of distinct induced subgraphs of $G$ which are isomorphic to $H$
\end{definition}

\begin{remark}
When one speaks of computing the number of copies of $H$ inside $G$, one is usually talking about counting the number of graph injections from $H$ to $G$ up to post-composition by automorphisms of $H$. This is the perspective we take in this work.\\
\end{remark}

The question of determining whether $\eta_H(G) > 0$ is known as the subgraph isomorphism problem. It is known, for general choices of $H$ and $G$, that the subgraph isomorphism problem is NP-complete \cite{Co}\cite{KOU}. The analogous induced subgraph isomorphism problem is also known to be NP-complete, although it is also known to be solvable in polynomial time in many instances \cite{Sy}. In this paper, we will be concerned with computing these two counting invariants across the members of certain families of graphs (see Theorem \ref{enumthm}).\\

After enumerative considerations, we next turn our attention to topological applications of the $\FI$-graph structure. Our first application is related to so called Hom-complex construction. Interest in these complexes originates from work of Lov\'asz \cite{Lo}, wherein similar spaces were used to resolve the Kneser conjecture. Babson and Koslov later showed that the the spaces used in Lov\'asz's work were specific examples of $\Hom$-complexes \cite{BK,BK2}. Following this, there has been some amount of interest in various topological aspects of these spaces (see \cite{Do,Do2} for some examples). For instance, it is known the every simplicial complex can be realized as the Hom-complex of some pair of graphs \cite{Do2}. In this paper, we will approach the Hom-complex from the perspective of representation stability.\\

\begin{definition}\label{homcomplexdef}
Let $T,G$ be graphs. A \textbf{multi-homomorphism} from $T$ to $G$ is a map of sets
\[
\alpha:V(T) \rightarrow \mathcal{P}(V(G))-\emptyset
\]
such that if $\{x,y\} \in E(T)$ then for all $x' \in \alpha(x)$ and all $y' \in \alpha(y)$, $\{x',y'\} \in E(G)$. The \textbf{Hom-complex of $T$ and $G$}, $\fiHom(T,G)$, is the simplicial complex whose simplicies are multi-homomorphisms from $T$ to $G$ such that $\alpha$ is a face of $\tau$ if and only if $\alpha(x) \subseteq \tau(x)$ for all $x \in T$.\\
\end{definition}

We will later construct large families of graphs $G_n$, indexed by the natural numbers, such that for any graph $T$, the complexes $\fiHom(T,G_n)$ are representation stable in the sense of Church and Farb (see Theorem \ref{homcomrepstab} and Definition \ref{repstable}).\\

Following this, we will spend some time proving facts about configuration spaces of graphs. The background for this material is detailed in the next section.\\

The final type of application we will concern ourselves with relates to spectral properties of graphs. More specifically, we will concern ourselves with eigenspaces and eigenvalues of adjacency and Laplacian matrices.\\

\begin{definition}
Let $G$ be a graph. The \textbf{adjacency matrix of $G$}, $A_G$, is the matrix whose columns and rows are labeled by vertices of $G$ and whose entries are defined by
\[
(A_G)_{(v,w)} := \begin{cases} 1 &\text{ if $\{v,w\} \in E(G)$}\\ 0 &\text{ otherwise.}\end{cases}
\]
The \textbf{Laplacian matrix of $G$}, $L_G$, is the difference $D_G - A_G$, where $D_G$ is the diagonal matrix whose entries display the degrees of the vertices of $G$.\\

The collection of eigenvalues of $A_G$ will be referred to as the \textbf{spectrum} of $G$.
\end{definition}

There are many things that one may immediately observe from the fact that $A_G$ and $L_G$ are real and symmetric. For instance:\\
\begin{enumerate}
\item $A_G$ and $L_G$ are diagonalizable. In particular, all of their eigenvalues are ordinary;
\item the eigenvalues of $A_G$ and $L_G$ are real. Therefore, they can be ordered as $\lambda_1 \geq \lambda_2 \geq \ldots \geq \lambda_{|V(G)|}$.\\\label{ordered}
\end{enumerate}

In our work, we will be largely concerned with the following two questions: Given certain natural families of graphs $G_n$, indexed by the natural numbers, how many distinct eigenvalues can $A_{G_n}$and $L_{G_n}$ have (as a function of $n$), and how do the multiplicities of these eigenvalues change with $n$. For instance, it is easily verified that the adjacency matrix of the complete graph $K_n$, with $n \geq 2$, will have distinct eigenvalues $n-1$ and $-1$ appearing with multiplicities 1 and $n-1$, respectively. In other words, so long as $n$ is sufficiently large, the complete graph $K_n$ has a fixed number of distinct eigenvalues, and the corresponding eigenspaces have dimensions which are polynomial in $n$. One of the main motivations for this paper is proving a framework which explains such behavior.\\

For references on graph spectra, see \cite{B,CDS,CRS,CRS2}.\\

\subsection{Configuration spaces of graphs}

\begin{definition}
Let $G$ be a graph. Then the \textbf{$m$-stranded configuration space of $G$} is the topological space
\[
\Conf_m(G) := \{(x_1,\ldots,x_m) \in G^m \mid x_j \neq x_i, i \neq j.\}
\]
\end{definition}

Configuration spaces of various topological spaces have a long history including work of McDuff \cite{McD}, and Church, Ellenberg, and Farb \cite{CEF}, among many others. Much of the literature is focused on the configuration spaces of manifolds. Recently, some attention has been given to the configuration spaces of graphs, due to their connections with robotics \cite{Gh}. Much of the newly emerging literature seems to indicate that these configuration spaces are heavily influenced by the combinatorics of the graph (see \cite{A,Gh,Gal,FS,R,Lu,CL}, for a small sampling). For instance, the following theorem of Abrams puts a very natural cellular structure on $\Conf_m(G)$, which depends highly on the vertices of $G$ of degree at least 3. Cellular models have also been proposed by \'Swiatkowski \cite{Sw}, Ghrist \cite{Gh}, L\"utgehetmann \cite{Lu}, and Wiltshire-Gordon \cite{WG}.\\

\begin{definition}
Let $G$ be a graph. The \textbf{$m$-th subdivision} of $G$ is the graph $G^{(m)}$ obtained from $G$ by adding $m-1$ vertices of degree 2 to every edge of $G$.\\
\end{definition}

\begin{theorem}[Abrams, \cite{A}]\label{cellularmodel}
Let $G$ be a graph, and let $\DConf_m(G)$ denote the sub-complex of the cubical complex $G^m$ comprised of cells of the form
\[
\sigma_1 \times \ldots \times \sigma_m
\]
where $\sigma_i$ is either an edge or vertex of $G$, and for each $i \neq j$,
\[
\sigma_i \cap \partial(\sigma_j) = \emptyset.
\]
Then $\DConf_m(G^{(m)})$ is homotopy equivalent to $\Conf_m(G^{(m)})$.\\
\end{theorem}

The original work of Abrams is more precise than the above, however it is sufficient for what follows. We note that for any graph $G$, $\Conf_m(G^{(m)})$ is identical to $\Conf_m(G)$. We observe that $\DConf_m(G)$ is the largest subcomplex of $G^m$ which avoids the diagonals $x_i = x_j$. Abrams' theorem therefore states that this complex will contain the same topological information as $\Conf_m(G)$ so long as there are enough vertices in $G$ such that every coordinate in a given configuration can fit on a single edge using only vertices.\\

It is often convenient to visualize the cells of $\DConf_m(G^{(m)})$ as living on the graph $G^{(m)}$. In such a visualization, we bolden the vertices and edges appearing in the cell on the graph $G^{(m)}$, and label the position in which they appear in the cell. For instance, Figure \ref{cellexample} shows a cell of $\DConf_2(G^{(2)})$ for a particular choice of $G$.\\

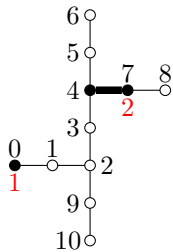
\begin{figure}
\begin{tikzpicture}
\tikzstyle{every node}=[draw,circle,fill=white,minimum size=4pt,
                            inner sep=0pt]
    \draw (0,0) node (0) [fill = black, label=above:$0$, label = below:\textcolor{red}{$1$}] {}
		      --(.5,0) node (1) [label=above:$1$]{}
					--(1,0) node (2) [label=right:$2$]{}
					--(1,.5) node (3) [label=left:$3$]{}
					--(1,1) node (4) [fill = black, label=left:$4$]{}
					--(1,1.5) node (5) [label=left:$5$]{}
					--(1,2) node (6) [label=left:$6$]{}
    			(4) -- (1.5,1) node (7) [fill = black, label=above:$7$, label=below:\textcolor{red}{$2$}]{}
					--(2,1) node (8) [label=above:$8$]{}
					(2) -- (1,-.5) node (9) [label=left:$9$]{}
					--(1,-1) node (10) [label=left: $10$]{};
		 \draw [line width=1mm] (4) -- (7);
\end{tikzpicture}
\caption{The cell $(v_0,e_{4,7})$.}\label{cellexample}
\end{figure}

Among the many incredible theoretical properties of configuration spaces of graphs is the precise computation of their Euler characteristic. The following result is due to Gal, and provides a large part of the motivation for this work.\\

\begin{theorem}[Gal, \cite{Gal}]\label{eulerchar}
Let $G$ be a graph, and let $\mathfrak{e}(t)$ denote the exponential generating function
\[
\mathfrak{e}(t) = \sum_{m \geq 0} \frac{\chi(\Conf_m(G))}{m!}t^m.
\]
Then,
\[
\mathfrak{e}(t) = \frac{\prod_{v \in V(G)}(1-(1-\mu(v))t)}{(1-t)^{|E(G)|}}.
\]
\end{theorem}

A theorem of Ghrist \cite{Gh} and  \'Swiatkowski \cite{Sw} implies that $H_i(\Conf_m(G)) = 0$ for all graphs $G$ and all $i$ larger than the number of vertices of $G$ of degree at least 3. In particular, it is independent of $m$. It follows from this fact, as well as the theorem of Gal, that the Betti numbers of $\Conf_m(G)$ should be expected to grow in $m$ like $m!$. Such growth precludes $\Conf_m(G)$ from being representation stable (see Definition \ref{repstable} Theorem \ref{repstab}). Looking again at the theorem of Gal, we see that the Euler characteristic of $\Conf_m(G)$, as a function of $m$, looks like $m!$ multiplied by a polynomial in invariants of $G$. In other words, the extreme growth in the Euler characteristic seems to be be primarily influenced by the number of points being configured, rather than the the graph $G$ itself.\\

One guiding philosophy of the present work is that if we fix the number of points begin configured, and instead allow the graph itself to vary, then the collection of spaces $\Conf_m(G_n)$ will be representation stable in the sense of Definition \ref{repstable}.\\

This philosophy has also appeared in recent work of L\"utgehetmann \cite{Lu}. Theorem \ref{repstabconf} extends the main theorem of that work.\\

\subsection{$\FI$-modules and representation stability}

The main tool we introduce in this paper are objects we refer to as $\FI$-graphs. Before working through the technical details of that construction, we must first discuss a key auxiliary concept: $\FI$-modules.\\

\begin{definition}
Let $\FI$ denote the category whose objects are the sets $[n] := \{1,\ldots,n\}$ finite sets and injections. An $\FI$-module is a (covariant) functor from $\FI$ to the category of $\R$ vector spaces.\\
\end{definition}

\begin{remark}
Note that most works in the literature allow $\FI$-modules to be valued in any module category over a commutative ring. For our purposes, we will mostly consider $\FI$-modules which are valued in vector spaces over $\R$. In certain areas of the paper, such as Lemma \ref{homcomplexfg}, we consider $\FI$-modules over $\Z$. Most of the definitions and theorems in this section work equally well in this case.\\
\end{remark}

Note that we will often write $V_n := V([n])$ and $f_\as := V(f)$. One should note that, for any $n$, the endomorphisms of $[n]$ in $\FI$ are precisely the permutations on $n$ letters, $\Sn_n$. Functoriality therefore implies that, for each $n$, $V_n$ is a representation of $\Sn_n$. This observation leads to the following perspective on $\FI$-modules: an $\FI$-module $V$ is a series of $\Sn_n$-representations, $V_n$, with $n$ increasing, which are compatible with one another through the action of the maps $V(f)$.\\

Just as with the study of vector spaces, it is often reasonable to restrict ones attention to those objects which are finitely generated in the appropriate sense. Before describing how such a condition can be applied to $\FI$-modules, we note that the category of $\FI$-modules and natural transformations is abelian. Indeed, one may define the usual abelian operations point-wise. In fact, one may very naturally define constructions such as direct sums and products, tensor products, symmetric products, etc. for $\FI$-modules.\\

\begin{definition}
An $\FI$-module $V$ is said to be \textbf{finitely generated in degree $\leq d$} if there is a finite set
\[
\{v_i\} \subseteq \sqcup_{n = 0}^d V_n
\]
which no proper submodule of $V$ contains. Equivalently, the set $\{v_i\}$ generates $V$ if, for all $n$, the vector space $V_n$ is spanned by the images of the $v_i$ under the various maps $f_\as$ induced by $V$ from injections of sets.\\
\end{definition}

Perhaps the most remarkable thing about finitely generated $\FI$-modules is that they exhibit a Noetherian property. The following was first proven by Snowden in \cite{Sn}, and later repoven by Church, Ellenberg, and Farb in \cite{CEF}.\\

\begin{theorem}[Snowden, \cite{Sn}; Church, Ellenberg, and Farb, \cite{CEF}]
Let $V$ be a finitely generated $\FI$-module. Then every submodule of $V$ is also finitely generated.\\
\end{theorem}

We will use the above Noetherian property to deduce various somewhat surprising combinatorial facts about $\FI$-graphs.\\

As one might expect, if $V$ is an $\FI$-module generated in degree $\leq d$, then it is not necessarily the case that submodules of $V$ are also generated in degree $\leq d$. Despite this, one may still conclude certain things about submodules of $V$ based on properties of $V$. For this reason, we introduce the following.\\

\begin{definition}
We say that a finitely generated $\FI$-module $V$ is \textbf{$d$-small} if $V$ is a subquotient of an $\FI$-module which is finitely generated in degree $\leq d$.\\
\end{definition}

\begin{proposition}[Church, Ellenberg, and Farb, \cite{CEF}]\label{tensorgen}
If $V$ is finitely generated in degree $\leq d$ and $W$ is finitely generated in degree $\leq e$, then
\begin{enumerate}
\item $V \oplus W$ and $V \times W$ are generated in degree $\leq \max\{d,e\}$;
\item $V \otimes W$ is generated in degree $\leq d+e$.
\end{enumerate}
\end{proposition}

The following list of properties are proven throughout \cite{CEF}.\\

\begin{theorem}[Church, Ellenberg, and Farb, \cite{CEF}]\label{repstab}
Let $V$ be an $\FI$-module. If $V$ is finitely generated then for all $n \gg 0$ and all injections $f:[n] \rightarrow [n+1]$,
\begin{enumerate}
\item $f_\as$ is injective;
\item the vector space $V_{n+1}$ is spanned as an $\Sn_{n+1}$-representation by $f_\as(V_n)$;
\item the $\Sn_n$-representation $V_n$ admits a decomposition of the form
\[
V_n = \bigoplus_{\lambda, |\lambda| \leq d} m_\lambda V(\lambda)_n
\]
where the coefficient $m_\lambda$ is independent of $n$ and $d$ is some constant independent of $n$ (see \cite{CEF} for details on the representations $V(\lambda)$). In particular, the multiplicity of the trivial representation in $V_n$ is eventually independent of $n$.
\item if $V$ is $d$-small, then there exists a polynomial $p_V(X)\in \Q[X]$ of degree $\leq d$ such that for all $n \gg 0$, $p_V(n) = \dim_\Q V_n$.\label{poly}\\
\end{enumerate}
\end{theorem}

The above will be used extensively in what follows.\\

The notion of representation stability was first introduced by Church and Farb in their seminal work \cite{CF}. From these beginnings the field has seen a boom in the literature and has been proven to be applicable to a large collection of subjects. For the purposes of this paper, we state the following definition, which is a modernized version of the original definition of Church and Farb.\\

\begin{definition}\label{repstable}
Let $X_\dt$ denote a functor from $\FI$ to the category of topological spaces. Then we say that $X_\dt$ is \textbf{representation stable} if for all $i \geq 0$ the $\FI$-module over $\Z$
\[
H_i(X_\dt;\Z)
\]
is finitely generated.
\end{definition}

It was famously proven by Church \cite{Ch}, and later reexamined by Church, Ellenberg, and Farb \cite{CEF}, that if $M$ is a compact orientable manifold with boundary of dimension at least two
\[
n \mapsto \Conf_n(M)
\]
is representation stable. We have already seen, however, that an analogous statement cannot be true if we replace $M$ with a graph (see the discussion following Theorem \ref{eulerchar}). We therefore change our approach and instead consider the functors
\begin{eqnarray}
n \mapsto \Conf_m(G_n) \label{functor}
\end{eqnarray}
where $m$ is fixed, and $G_\dt$ is a particularly nice $\FI$-graph (see the statement of Theorem \ref{repstabconf}). The main theorem of this paper can be restated to say that in this case the functor (\ref{functor}) is representation stable. Our approach will be largely combinatorial, and we will use structural facts about $\FI$-graphs as well as the cellular model of Theorem \ref{cellularmodel}. This is in contrast to the work of L\"utgehetmann, which proves that $n \mapsto \Conf_m(G_n)$ is representation stable for a particular choice of $G_\dt$ (see Example \ref{lutex}) using very topological methods. We will find that our method provides a stronger bound on the degree of the polynomial encoding the Betti numbers in this case, while L\"utgehetmann's method provides bounds on the degree of generation of the $\FI$-modules $H_i(\Conf_m(G_\dt))$.\\

\section{FI-graphs}

\subsection{Definitions and examples} \label{def}

The primary objective of this section is to provide a framework through which one can study families of graphs in the spirit of Kneser graphs and their generalizations. Recall that, for any fixed integers $n \geq k$, one defines the Kneser graph $KG_{n,k}$ as the graph whose vertices are labeled by $k$-element subsets of $[n]$, and whose edges connect disjoint sets.\\

It is clear that for each $n$, elements of $\Sn_n$ act on $KG_{n,k}$ by graph automorphisms. What is perhaps more subtle, is that if $f:[n] \hookrightarrow [m]$ is any injection, then there is an induced map of graphs
\[
KG(f):KG_{n,k} \rightarrow KG_{m,k}
\]
Looking back at the definition of $\FI$-modules, one is therefore motivated to make the following definition.\\

\begin{definition}
An \textbf{$\FI$-graph} is a functor from the category $\FI$ to the category $\mathbf{Graph}$ of (simple) graphs. We will usually denote an $\FI$-graph by $G_\dt:\FI \rightarrow \mathbf{Graph}$. We will use $G(f)$ to denote the induced maps of $G_\dt$.\\
\end{definition}

While the above definition captures the core of the above discussion, it is still a bit too general for our purposes. For instance, if
\[
G_0 \subseteq G_1 \subseteq G_2 \subseteq \ldots
\]
is any chain of graphs, then we may define an $\FI$-graph by setting the $\Sn_n$-action to be trivial for each $n$, and having the transition maps be the given inclusions. An arbitrary chain of graphs like the above can become rather complicated, and there won't necessarily be any way to gather meaningful information above the invariants of any $G_n$ from those that came before it. What is needed is some notion of finite generation for $\FI$-graphs. For this purpose, we define the following.\\

\begin{definition}
Let $G_\dt$ be an $\FI$-graph. We say that $G_\dt$ is \textbf{vertex-stable with stable degree $\leq d$} if for all $n \geq d$, and every vertex $v \in V(G_{n+1})$ there exists some vertex $w \in V(G_n)$ and some injection $f:[n] \hookrightarrow [n+1]$ such that $G(f)(w) = v$.\\
\end{definition}

That is, an $\FI$-graph is vertex-stable with stable degree $\leq d$ if for each $n > d$, every vertex in $G_{n}$ is in the image of one of the transition maps. Informally, no `new' vertices appear after the graph $G_d$, up to symmetric group actions. \\

We will find that this fairly simple combinatorial condition is sufficient to prove a plethora of facts about the graphs $G_n$. Before we delve into these details, we first introduce the various examples which motivated this paper. In most of these examples, vertices are labelled by elements of $[n] = \{1,2,\dots,n\}$ or by ordered or unordered tuples of such elements. The symmetric group $\Sn_n$ acts on such vertices by acting on each element individually.\\

\begin{example}\label{kneserex}
For any fixed $k \geq 0$, the Kneser graphs $KG_{\dt,k}$ form a vertex-stable $\FI$-graph with stable degree $k$ (or stable degree $1$ if $k = 0$). The same can therefore be said about the complete graphs $K_\dt = KG_{\dt,1}$.\\

More generally, if $n,k,r$ are fixed integers, then we define the generalized Kneser graph $KG_{n,k,r}$ to have vertices labeled by subsets of $[n]$ of size $k$, edges connecting subsets whose intersection has size $\leq r$. In particular, $KG_{n,k} = KG_{n,k,0}$. The generalized Kneser graphs $KG_{\dt,k,r}$ form a vertex-stable $\FI$-graph for each fixed $k$ and $r$, again with stable degree $k$.\\
\end{example}

\begin{example}
For any fixed $k \geq 0$, we can define a variant of the Kneser graph, which we denote $KG_{n,\leq k}$. The vertices of $KG_{n,\leq k}$ will be labeled by subsets of $[n]$ of size at most $k$, and the edges will connect disjoint subgraphs, just as was the case with the Kneser graph. Because self-loops are forbidden, we do not connect the empty set to itself.\\

Note that, for each $n$, the symmetric group action on $KG_{n,\leq k}$ is not transitive. Despite this, the collection $KG_{\dt,\leq k}$ still form a vertex-stable $\FI$-graph with stable degree $k$. It will be useful to consider the orbits of vertices under the symmetric group actions. Our examples tend to have few orbits for the sake of being simple examples, but this is not a restriction on general $\FI$-graphs. \\
\end{example}

\begin{example}
For any fixed $k \geq 0$, the complete bipartite graphs $K_{\dt,k}$ form a vertex-stable $\FI$-graph with stable degree $1$. Here, our transition maps and permutations fix the vertices in the part of size $k$. It follows that the series of star graphs, Star$_\dt = K_{\dt,1}$ form a vertex-stable $\FI$-graph.\\
\end{example}

\begin{example}
For any fixed $n,k \geq 0$, define the Johnson graph $J_{n,k}$ as that whose vertices are labeled by subsets of $[n]$ with size $k$, and whose edges connect subsets with intersection size $k-1$. Then $J_{\dt,k}$ naturally forms a vertex-stable $\FI$-graph with stable degree $k$.\\
\end{example}

\begin{example}\label{latticeex}
Recall that the $n$-cube graph $Q_n$ is defined to be the 1-skeleton of the $n$-dimensional hypercube. This collection cannot be endowed with the structure of a vertex-stable $\FI$-graph, as its number of vertices grows too fast (see Theorem \ref{degreegrow}). There is, however, a variation of the $n$-cube graph which can be endowed with the structure of a finitely generated $\FI$-graph.\\

For fixed $n,k \geq 0$, let $Q_{n,k}$ denote the graph whose vertices are ordered $k$-tuples of elements of $[n]$, where two vertices are connected if they differ in only one coordinate. This graph is sometimes called the $k$-lattice graph of characteristic $n$. The cubic lattice graph of characteristic $n$ is notable in that it can be entirely characterized by certain simple combinatorial properties (see \cite{La}). For our purposes, we simply note that for any fixed $k$ the family $Q_{\dt,k}$ can be endowed with the structure of a vertex-stable $\FI$-graph. Indeed, let $n > k$, and let $(i_1,\ldots,i_k)$ be a vertex of $Q_{n+1,k}$. Because $k < n$, we know that there is some integer $l \in [n]$ such that $l \neq i_j$ for any $j$. Then $(i_1,\ldots,i_k)$ is in the image of the transition map induced by the injection $f:[n] \hookrightarrow [n+1]$ given by,
\[
f(x) = \begin{cases} n+1 &\text{ if $x = l$,}\\ x &\text{ otherwise.}\end{cases}
\]
This $\FI$-graph has stable degree $k$. 
\end{example}

\begin{example}\label{lutex}
Our next example appears in earlier work of L\"utgehetmann \cite{Lu}. Let $G,H$ be any pair of pointed graphs. Then we can construct a new graph by wedging $G$ with $H$ $n$-times,
\[
G_n := G \bigvee H^{\vee n}
\]
Then we may endow $G_n$ with the structure of an $\FI$-graph by having the symmetric group act by permuting the factors of $H$. This $\FI$-graph has stable degree $1$.\\
\end{example}

The examples thus far have been quite regular, in the sense that for each $n$, the construction of the vertices and edges of the graph $G_n$ has been the same. It is worth examining how this can be varied, particularly because results later in this section will limit how wild such variation can be.

\begin{example}\label{zeroedgeex}
Let $G_\dt$ be an $\FI$-graph, and modify it by removing all edges from each $G_i$, for $i = 1$ to $k-1$.
\end{example}

\begin{example}\label{zerographex}
Let $G_\dt$ be an $\FI$-graph, and modify it by replacing each $G_i$ by the empty graph, for $i = 1$ to $k-1$.
\end{example}

While Examples \ref{zeroedgeex} and \ref{zerographex} remove vertices and edges from graphs in the first few degrees, this cannot necessarily be done in later degrees. The transition maps are permitted to map pairs of vertices not connected by an edge to pairs of vertices connected by an edge, but not the reverse. Two vertices joined by an edge may not map to the same vertex, because there cannot be an edge from this vertex to itself. For instance, if $G_n$ contains a complete graph on $k$ vertices then $G_{n+1}$ also contains a complete graph on $k$ vertices.\\

Disjoint unions of $\FI$-graphs are $\FI$-graphs, and it is possible to increase the number of copies from a certain point onwards.

\begin{example}\label{duplicateex}
Let $G_\dt$ be any $\FI$-graph. Fix a positive integer $k$, and create a new $\FI$-graph $H_\dt$ as follows. For $i < k$, the graph $H_i$ is equal to $G_i$. For $i \geq k$, the graph $H_i$ is a disjoint union of two copies of $G_i$. For concreteness, color vertices and edges in one of these subgraphs red and in the other, blue. The action of $\Sn_n$ preserves the color of vertices. Transition maps preserve the color of vertices and take uncolored vertices to red vertices. This $\FI$-graph has stable degree $k$.
\end{example}

Example \ref{duplicateex} did not need the two graphs to be the same --- the new graphs introduced from degree $k$ could have been the respective components of any $\FI$-graph.\\

It is also possible to decrease the number of components. This does require the use of transition maps which are not injective.

\begin{example}\label{collideex}
Let $G_\dt$ be any $\FI$-graph. Fix a positive integer $k$, and create a new $\FI$-graph $H_\dt$ as follows. For $i < k$, the graph $H_i$ is a disjoint union of two copies of $G_i$. Color vertices and edges in one of these subgraphs red and in the other, blue. For $i \geq k$, the graph $H_i$ is equal to $G_i$. The action of $\Sn_n$ preserves the color of vertices. Transition maps preserve the color of vertices if their image is in $G_i$ with $i < k$, and forget colors otherwise.
\end{example}

An $\FI$-graph may be modified by changing the times at which the various `types' of edges begin to appear, as in the following variant of the Kneser graph.

\begin{example}\label{edgeorbitex}
Let the vertex set of $G_n$ be indexed by $r$-tuples of elements of $[n]$, and let $a_0$ to $a_r$ be $r+1$ fixed positive integers. In $G_n$, there is an edge between two vertices which share exactly $k$ elements if and only if $n \geq a_k$. 
\end{example}

Example \ref{edgeorbitex} could be generalized further by taking the vertices to be ordered $r$-tuples, or tuples with only limited ordering information, in which case there would be more edge orbits. For instance, if $r=2$, then there are five orbits of edges rather than three in the unordered case --- between pairs of vertices $((a,b),(a,c))$,$((a,b),(c,b))$,$((a,b),(b,c))$,$((a,b),(c,a))$, and $((a,b),(c,d))$.\\

The next example fails to be an $\FI$-graph in a subtle way. If it was an $\FI$-graph, it would violate Theorem \ref{inducedtheorem}.

\begin{example}\label{orbitmergenonex}
For each $i \neq 2$, let $G_i$ be the complete graph on the vertex set $[i]$, with the natural symmetric group action. Let $G_2$ have vertex set $\{1,2,3\}$, with edges $13$ and $23$. Transition maps from $G_n$ to $G_{n+1}$ are the identity map postcomposed with any element of $\Sn_{n+1}$.\\

However, for $G_\dt$ to be an $\FI$-graph, injections from $[2]$ to $[n]$ for $n > 3$ must induce maps from $G_2$ to $G_n$, and it is here that this construction fails --- a map from $[2]$ to $[n]$ can't specify where each of the three vertices of $G_2$ is sent. This is perhaps a surprising failure, because transition maps from $G_n$ to $G_{n+1}$ can be defined properly, and it is only longer-range maps which fail.\\

If rather than $\FI$ we were working over a category where maps from $[2]$ to $[n]$ were instead a sequence of maps from $[2]$ to $[3]$ to $[4]$ and so on, then this construction would not fail in this way, and so over this category, the analogue of Theorem \ref{inducedtheorem} is false.
\end{example}

\subsection{Vertex-stability and its consequences}

While it is clearly the case that the examples of Section \ref{def} are vertex-stable, one might also note that these cases seem to have much more structure than this. For instance, it is natural to go a step further and make the following definitions:\\
\begin{definition}~
\begin{enumerate}
\item an $\FI$-graph is \textbf{eventually injective} if for $n \gg 0$, the transition maps of $G_\dt$ are injective;\\
\item an $\FI$-graph is \textbf{eventually induced} if for $n \gg 0$, the image of any transition map is an induced subgraph;\\
\item an $\FI$-graph is \textbf{edge-stable} with edge-stable degree $\leq k$ if for $n \geq k$ and any $\{x,y\} \in E(G_{n})$ there is an edge $\{v,w\} \in V(G_k)$ and an injection $f:[k] \hookrightarrow [n]$ such that $G(f)(v) = x$ and $G(f)(w) = y$;\\
\item an $\FI$-graph is \textbf{$r$-vertex-stable} if for all $n \gg 0$, and any collection of $r$ vertices of $V(G_{n+1})$, $\{x_1,\ldots,x_r\}$, there is a collection of vertices of $G_n$, $\{v_1,\ldots,v_r\}$, and an injection $f:[n] \hookrightarrow [n+1]$, such that $G(f)(v_i) = x_i$ for each $i$.\\
\end{enumerate}
\end{definition}

These stability properties may occur at quite different times, and at different times to vertex stability. Example \ref{collideex} is injective only from degree $k$ onwards, Example\ref{duplicateex} is vertex-stable and edge-stable from degree $k$ onwards, and Example \ref{edgeorbitex} is vertex-stable in degree $r$, but edge-stable only once the degree is greater than all of $a_0$ through $a_r$. \\

The Kneser graphs $KG_{\dt,k}$ (Example \ref{kneserex}) are vertex-stable in degree $k$, edge-stable in degree $2k$, and $r$--vertex stable in degree $rk$. In contrast, the lattice graphs $Q_{\dt,k}$ (Example \ref{latticeex}) are vertex-stable in degree $k$, edge-stable in degree $k+1$, and $r$--vertex stable in degree $rk$. \\

It is left to the reader to verify that all of the examples of the previous section satisfy each of the above conditions. Somewhat miraculously, it turns out that this is not a coincidence.\\

\begin{theorem}\label{fgprops}
Let $G_\dt$ be a vertex-stable $\FI$-graph. Then:
\begin{enumerate}
\item $G_\dt$ is $r$-vertex stable for all $r \geq 1$;
\item $G_\dt$ is edge-stable;
\item $G_\dt$ is eventually injective and induced.\\
\end{enumerate}
\end{theorem}

It is worth noting that vertex stability is strictly stronger than edge stability, as shown by the following example.\\

\begin{example}
For each $n$, let $G_n$ be the union of the complete graph $K_n$ and $n$ isolated vertices. The symmetric group $\Sn_n$ acts naturally on the complete graph and fixes each of the other vertices. This $\FI$-graph is edge-stable in degree 2, but is not vertex-stable.\\
\end{example}

Edge stability may happen either before or after $2$-vertex stability, because edge stability includes only pairs of vertices which are connected by edges, but it is possible for edges to not appear until long after any pair of vertices are contained in the image of some transition map. Example \ref{edgeorbitex} is $2$-vertex stable in degree $2r$, but is not edge-stable until the degree equal to the maximum of the $a_i$. \\

Before we prove Theorem \ref{fgprops}, it will be useful to us to rephrase the above properties in terms of finite generation of certain $\FI$-modules.\\

\begin{definition}
Let $G_\dt$ denote an $\FI$-graph, and let $r \geq 1$ be fixed. We write
\[
\R\binom{V(G_\dt)}{r}
\]
to denote the $\FI$-module whose evaluation at $[n]$ is the $\R$ vector space with basis indexed by collections of $r$ vertices of $G_n$. We will often write $\R V(G_\dt) := \R\binom{V(G_\dt)}{1}$. Note that the image of a collection of $r$ vertices under a transition map may not be a collection of $r$ vertices if this transition map is not injective on vertices. In this case we simply declare the map to be zero on this collection. Similarly, we define $\R E(G_\dt)$ to be the $\FI$-module whose evaluation at $[n]$ is the $\R$ vector space with basis indexed by the edges of $E(G_n)$.\\
\end{definition}

\begin{lemma}\label{fimodchar}
Let $G_\dt$ be an $\FI$-graph.
\begin{enumerate}
\item $G_\dt$ is vertex-stable with stable degree $\leq d$ if and only if $\R V(G_\dt)$ is finitely generated in degree $\leq d$.
\item $G_\dt$ is eventually injective if and only if the transition maps of $\R V(G_\dt)$ are eventually injective.
\item $G_\dt$ is edge-stable with edge-stable degree $\leq d$ if and only if $\R E(G_\dt)$ is finitely generated in degree $\leq d$. 
\item $G_{\dt}$ is $r$-vertex-stable if and only if $\R\binom{V(G_\dt)}{r}$ is finitely generated.\\
\end{enumerate}
\end{lemma}

\begin{remark}
Note that this lemma is critically dependent on the assumption that $G_n$ has finitely many vertices and edges for each $n$. For instance, consider the collection of infinite graphs
\[
V(G_n) := \N,\hspace{1cm} E(G_n) := \{\{1,2\},\{2,3\},\ldots, \{n-1,n\}\}
\]
We can introduce an $\FI$-structure on $G_\dt$ by having the symmetric group act trivially. Then it is clear that $\R V(G_n)$ is not finitely generated, despite $G_\dt$ being ``vertex-stable'' in some sense. Also note that the collection $G_\dt$ is not edge-stable in this case, seemingly violating Theorem \ref{fgprops}.\\
\end{remark}

This lemma is the key piece in the proof of Theorem \ref{fgprops}.\\

\begin{proof}[Proof of Theorem \ref{fgprops}]
To begin, Lemma \ref{fimodchar} implies that we must show that $\R\binom{V(G_\dt)}{r}$ is finitely generated. We note that there is a surjection of $\FI$-modules
\[
\Sym^r(\R V(G_\dt)) \rightarrow \R\binom{V(G_\dt)}{r} 
\]
Indeed, this is induced by the assignments
\[
x_1 \otimes \ldots \otimes x_r \mapsto \begin{cases} \{x_1,\ldots,x_r\} & \text{ if $x_i \neq x_j$ for $i \neq j$}\\ 0 & \text{ otherwise.}\end{cases}
\]
By assumption $\R V(G_\dt)$ is finitely generated, whence the same is true of $(\R V(G_\dt))^{\otimes r}$ by Proposition \ref{tensorgen}. The symmetric power $\Sym^r(\R V(G_n))$ is a quotient of $(\R V(G_\dt))^{\otimes r}$, and is therefore finitely generated as well. This concludes the proof.\\

The second statement follows from the Noetherian property as well as the inclusion
\[
\R E(G_\dt) \hookrightarrow \R\binom{V(G_\dt)}{2}.
\]

Eventual injectivity follows from Theorem \ref{repstab}.\\

By definition, $G_\dt$ is eventually induced if and only if for any pair of vertices of $G_n$, $\{x,y\} \notin E(G_n)$, and any injection $f:[n] \hookrightarrow [n+1]$, $\{f_\as(x),f_\as(y)\} \notin E(G_{n+1})$. For each $n$, let $\Ob_n$ denote the set of $\Sn_n$-orbits of pairs of vertices in $G_n$. Note that $\Ob_n$ may be partitioned into two subsets, depending on whether or not pairs in the orbit correspond to edges or not. Further note that the transition maps of $G_\dt$ will send an ``edge'' orbit to an edge orbit. On the other hand, the third part of Theorem \ref{repstab} implies that $|\Ob_n|$ is eventually independent of $n$, as it is equal to the multiplicity of the trivial representation in $\R\binom{V(G_n)}{2}$. This shows that non-edged orbits will eventually map exclusively into non-edged orbits, as desired.\\
\end{proof}

\begin{remark}\label{small}
Proposition \ref{tensorgen}, and the above proof, together imply that $\R\binom{V(G_\dt)}{r}$ is generated in degree $\leq r d$, where $d$ is the generating degree of $\R V(G_\dt)$. In particular, $\R E(G_\dt)$ is $2d$-small.\\
\end{remark}

\begin{remark}\label{rvertexstab}
It is possible to prove one part of Theorem \ref{fgprops} directly. Consider any set of $k$ vertices $v_1$ through $v_k$ in $G_n$, for $n \geq kr$. Each $v_i$ is in the image of a transition map from $G_k$ to $G_n$, and each of these transition maps is induced by an injection from $[k]$ to $[n]$. Let $f_1$ through $f_k$ be these injections. Take $f$ to be an injection from $[kr]$ to $n$ whose image includes the image of each $f_i$. Then each $f_i$ factors through $f$, so each $v_i$ is in the image of the transition map induced by $f$. This completes the proof.\\

The proof of $r$--vertex stability in Theorem \ref{fgprops} relies on the tensor product of finitely generated $\FI$-modules being finitely generated. The proof of this fact may be made explicit, and this is what lies behind the proof given above.\\
\end{remark}

Theorem \ref{fgprops} gives us a method for constructing vertex-stable $\FI$-graphs from given vertex-stable $\FI$-graphs.\\

\begin{definition}
Let $G$ be a graph. The \textbf{line graph} of $G$, $L(G)$, is the graph whose vertices are labeled by the edges of $G$ such that two vertices are connected if and only if the corresponding edges of $G$ share an end point.\\
\end{definition}

Line graphs have been studied extensively. One avenue of research is the question of how much of the graph $G$ can be determined by studying its line graph. A celebrated theorem of Whitney \cite{W} implies that the line graph almost always uniquely determines the original graph. Indeed, the only exception to this is the fact that $L(K_3) = L(K_{3,1})$. Algebraically, one is also interested in the question of deciding when a line graph is determined by its spectrum (See, for instance, \cite{H} or Chapter 1.3 of \cite{CRS}).\\

\begin{corollary}
Let $G_\dt$ denote a vertex-stable $\FI$-graph. Then the collection of line graphs $L(G_\dt)$ can be endowed with the structure of a vertex-stable $\FI$-graph.\\
\end{corollary}

\begin{proof}
This result follows immediately from Theorem \ref{fgprops} and the definition of the line graph.\\
\end{proof}

\begin{remark}
We note that we the line graph $L(K_n)$ is isomorphic to the Johnson graph $J_{n,2}$. The line graphs of the complete bipartite graphs $K_{n,m}$ have been studied (see, for instance, \cite{M} or the references in \cite{CRS}), and are sometimes referred to as the \textbf{rook graphs}, as they can be thought of as encoding legal rook moves on an $m \times n$ chess board.\\
\end{remark}

\subsection{Determining when the induced property begins}

Theorem \ref{fgprops} implies that all $\FI$-graphs are eventually induced. In this section we consider the question of bounding when this behavior begins. To begin we impose the following technical condition on the $\FI$-graph $G_\dt$. We will see this condition return again when we consider configuration spaces of graphs.\\

\begin{definition}
We say an $\FI$-graph $G_\dt$ is \textbf{torsion-free} if for all injections $f:[n] \hookrightarrow [m]$ the transition map $G(f)$ is injective.\\
\end{definition}

Most of the examples in Section \ref{def} are torsion-free. Example \ref{collideex} is not torsion-free.

\begin{remark}
We say an $\FI$-module is \textbf{torsion-free} if all of its transition maps are injective. The above definition is intended to emulate this.\\
\end{remark}

Theorem \ref{fgprops} insists that vertex-stability implies edge stability. In particular, at some point the transition maps of a vertex-stable $\FI$-graph will contain every edge in the union of their respective images. It is therefore natural for one to guess that it will be at this point that the image of these transition maps must be induced. We do indeed find this to be the case for torsion-free $\FI$-modules.\\

\begin{theorem}\label{inducedtheorem}
Let $G_\dt$ be a torsion-free vertex-stable $\FI$-graph with edge stable degree $\leq d_E$. Then for any $n \geq d_E$ and any injection $f:[n] \hookrightarrow [n+1]$ the image of the transition map $G(f)$ is an induced subgraph of $G_{n+1}$.\\
\end{theorem}

While it might seem natural for there to be some kind of pigeon-hole or counting argument for the above theorem, such an argument has thus far eluded the authors. Just like much of the rest of this work, we instead prove Theorem \ref{inducedtheorem} through the algebra of $\FI$-modules. To begin, we must rephrase the eventually induced property in the language of $\FI$-modules.\\

\begin{definition}
The \textbf{coinvariants functor} $\Phi$ from $\FI$-modules to graded $\R[x]$-modules is defined by
\[
\Phi(V)_n := V_n \otimes_{\Sn_n} \R
\]
Multiplication by $x$ is induced by the action of the transition maps.
\end{definition}

In the setting of $\FI$-graphs and their associated $\FI$-modules, the coinvariants functor takes a particularly nice form.\\

Recall that we define $\R\binom{V(G_\dt)}{2}$ to be the $\FI$-module encoding pairs of vertices of $G_\dt$. The coinvariants of $\R\binom{V(G_\dt)}{2}$ can be constructed in the following way. We define $\Phi$ to be the graded $\R[x]$-module for which $\Phi_n$ is the free $\R$ vector space with basis indexed by the orbits of the symmetric group action on pairs of vertices of $G_n$. For each $n$ we may define $\iota_n:[n] \hookrightarrow [n+1]$ to be the standard inclusion. Then $G(\iota_n)$ induces a map between the orbits of pairs of vertices of $G_n$ and those of $G_{n+1}$. Multiplication by $x$ in the module $\Phi$ will be defined by this map.\\

\begin{lemma}\label{torsionlemma}
Let $V$ be a finitely generated $\FI$-module. If $V$ is torsion-free as an $\FI$-module, then $\Phi(V)$ is torsion-free as a $\R[x]$-module.\\
\end{lemma}

\begin{proof}
This follows from the fact that coinvariants are exact over fields of characteristic 0.\\
\end{proof}

This lemma is the key piece needed to prove Theorem \ref{inducedtheorem}.\\

\begin{proof}[Proof of Theorem \ref{inducedtheorem}]
Let $G_\dt$ be a torsion-free vertex-stable $\FI$-graph, and assume that $G_\dt$ has edge stable degree $\leq d_E$. Assume by way of contradiction that there is some $n \geq d_E$ such that the image of any transition map $G(f):G_n \rightarrow G_{n+1}$ is not an induced subgraph. This implies that there is some pair of vertices $\{v_1,v_2\}$ in $G_n$, which are not connected by an edge, while $G(f)(\{v_1,v_2\})$ is an edge of $G_{n+1}$. On the other hand, because $n \geq d_E$, there must be some transition map $G(h)$, as well as some edge $e \in E(G_n)$ such that $G(h)(e) = G(f)(\{v_1,v_2\})$. We may apply some element of $\Sn_{n+1}$ to conclude the following: The transition map $G(f)$ must map some non-edge of $G_n$, as well as some edge of $G_n$, to the same $\Sn_{n+1}$ orbit on the pairs of vertices of $G_{n+1}$. In particular, this would imply that the coinvariants of $\R \binom{V(G_\dt)}{2}$ has torsion. This contradicts Lemma \ref{torsionlemma}.\\
\end{proof}

\section{Applications}

\subsection{Enumerative consequences of vertex-stability}

We begin this section by revisiting the invariants $\eta_H$ and $\eta_H^{ind}$ for some fixed graph $H$. In particular, if $G_\dt$ is a vertex-stable $\FI$-graph, we consider the functions
\[
n \mapsto \eta_H(G_n) \text{ and } n \mapsto \eta_H^{ind}(G_n).
\]
Our primary result in this direction is the following.\\

\begin{theorem}\label{enumthm}
Let $G_\dt$ be a vertex-stable graph with stable degree $\leq d$. Then for any graph $H$ there exists polynomials $p_H(X), p_H^{ind}(X) \in \Q[X]$ of degree $\leq d\cdot|V(T)|$ such that for all $n \gg 0$,
\[
p_H(n) = \eta_H(G_n) \text{ and } p_H^{ind}(n) = \eta_H^{ind}(G_n)
\]
\end{theorem}

\begin{proof}
We will count the number of graph injections from $H$ to $G_\dt$. This quantity is a constant multiple of $\eta_H(G_\dt)$, and it is therefore sufficient to count. Let $V^H_\dt$ denote the $\FI$-module whose evaluation at $[n]$ is the $\R$ vector space with basis indexed by the distinct copies of $H$ inside $G_n$. To make sure the transition maps are well defined, we will set $V^H_{n} = 0$ before the point where the transition maps of $G_\dt$ are both induced and injective. We therefore see that $V^H_\dt$ can be realized as a submodule
\[
V^H_\dt \hookrightarrow (\R V(H))^{\otimes |V(T)|}
\]
Proposition \ref{tensorgen} implies that $V^H_\dt$ is $(d\cdot|V(T)|)$-small, and Theorem \ref{repstab} implies the existence of our desired polynomial. The proof for the induced case is the same.\\
\end{proof}

\begin{example}
Looking the $\FI$-graph $K_\dt$ of complete graphs, the above result is clear. $H$ cannot appear in $K_n$ when $n < |V(H)|$. If we call $\gamma_H$ the number of copies of $H$ in $K_{|V(H)|}$, then
\[
p_H(n) = \binom{n}{|V(H)|}\gamma_H
\]
The content of Theorem \ref{enumthm} is that this behavior is common to all vertex-stable $\FI$-graphs. The examples of the previous section illustrate that vertex-stable $\FI$-graphs can be fairly diverse, and so this might come as a bit of a surprise.\\

Fix $k \geq 2$ and let $KG_{\dt,k}$ be the $\FI$-graph which encodes the Kneser graphs. In this case we may easily count the number of triangles which appear in $KG_{n,k}$. Indeed, to form a triangle, one needs to provide three mutually disjoint subsets of $[n]$ of size $k$. It follows that
\[
p_{K_3}(n) = \frac{\binom{n}{3k}\binom{3k}{k,k,k}}{6}.
\]
Note that if we take the usual convention that $\binom{n}{k} = 0$ whenever $n < k$, then the above polynomial agrees with $\eta_{K_3}(KG_{n,k})$ for all $n \geq 0$. If we instead try to count the number of occurrences of the graph which looks like the letter $H$, things get considerably more complicated. Despite the seeming drastic increase in difficulty, Theorem \ref{enumthm} assures us that the value of $\eta_H(KG_{n,k})$ must (eventually) agree with a polynomial, and that this polynomial will have degree at most $6k$.\\
\end{example}

As an immediate corollary to the above, we find that vertex-stable $\FI$-graphs have very controlled growth in their vertices and edges, as well as in the degrees of their vertices.\\

\begin{corollary}\label{degreegrow}
Let $G_\dt$ be a vertex-stable $\FI$-graph. Then the following functions are each equal to a polynomial for $n \gg 0$:
\begin{enumerate}
\item $n \mapsto |V(G_n)|$;
\item $n \mapsto |E(G_n)|$;
\item $n \mapsto \delta(G_n)$;
\item $n \mapsto \Delta(G_n)$.\\
\end{enumerate}
\end{corollary}

\begin{proof}
The first two statements follow from a direct application of Theorem \ref{enumthm} with $H$ being an isolated vertex and a single edge, respectively.\\

For the final two statements, we prove a more general statement. Fix $m \gg 0$, and let $v \in V(G_m)$. Then every vertex in the $\Sn_m$ orbit of $v$, which we denote $\Ob_v(m)$, has the same degree. Let $\R \Ob_v(\dt)$ denote the submodule of $\R V(G_\dt)$ generated by $v$, and for $n \geq m$ let $\mu(\Ob_v(n))$ denote the order of any (and therefore all) vertices in $\Ob_v(n)$. We will prove that the map
\[
n \mapsto \mu(\Ob_v(n))
\]
is equal to a polynomial. To see that this implies the final two statements of our corollary, note that the final part of Theorem \ref{repstab} implies that the total number of distinct orbits of $V(G_n)$ is eventually independent of $n$. Because non-equal polynomials are only permitted to be equal at finitely many points, the above implies that there is a well defined polynomial which outputs the smallest (or largest) degree of a vertex.\\

To prove our more general claim, we need a bit of notation. We will write $\R E(\Ob_v(\dt))$ for the submodule of $\R E(G_\dt)$ whose $n$-th piece is spanned by edges whose both end points are in $\Ob_v(n)$. We will also write $N(\Ob_v(n))$ to denote the subgraph of $G_n$ comprised of all vertices and edges that one may encounter by beginning at a vertex in $\Ob_v(n)$ and moving along any single edge adjacent to it. Put another way, $N(\Ob_v(n))$ is the neighborhood graph on the vertex set $\Ob_v(n)$. By setting $N(\Ob_v(n)) = \emptyset$ whenever $n < m$, we see that $N(\Ob_v(\dt))$ is actually a vertex-stable $\FI$-graph. Therefore, by the second part of this corollary,
\[
n \mapsto |E(N(\Ob_v(n)))|
\]
is eventually equal to a polynomial. On the other hand, we may count the set $|E(N(\Ob_v(n)))|$ in the following alternative way,
\[
|E(N(\Ob_v(n)))| = \mu(\Ob_v(n)) \cdot |\Ob_v(n)| - |E(\Ob_v(n))|.
\]
In other words, if we sum the degrees of all vertices in $\Ob_v(n)$, then we would have counted each edge in $E(\Ob_v(n))$ exactly twice. Because $\R \Ob_v(\dt)$ is a submodule of $\R V(G_\dt)$, we know that its dimension is eventually equal to a polynomial. A similar statement can also be made about $|E(\Ob_v(n))|$. Solving for $\mu(\Ob_v(n))$, we find that it is equal to a rational function for $n$ sufficiently large. However, the only rational functions which can take integral values at all sufficiently large integers are polynomials. This concludes the proof.\\
\end{proof}

Another consequence of vertex stability that one may deduce is related to finite walks in the graph $G_n$. Recall that a \textbf{walk of length $r$} in a graph $G$ is a tuple of vertices of $G$, $(v_0,\ldots,v_r)$, such that for each $i$, $\{v_i,v_{i+1})$ is an edge of $G$. We say that a walk is \textbf{closed} if $v_r = v_0$.\\

\begin{theorem}\label{walks}
Let $G_\dt$ be a vertex-stable $\FI$-graph with stable degree $\leq d$. Then the following functions are each equal to a polynomial of degree $\leq rd$ for $n \gg 0$ and any fixed $r \geq 0$:
\begin{enumerate}
\item $n \mapsto |\{\text{walks in $G_n$ of length $r$}\}|$;
\item $n \mapsto |\{\text{closed walks in $G_n$ of length $r$}\}|$.
\end{enumerate}
\end{theorem}

\begin{proof}
Our strategy here is similar to the strategy of much of the rest of the paper. Encode the objects we hope to count as the dimension of some vector space, and use the Noetherian property to prove that the collection of all these vector spaces form a finitely generated $\FI$-module. Let $W_r(G_\dt)$ denote the $\FI$-module for which $W_r(G_n)$ is the formal vector space spanned by walks of length $r$ in $G_n$. Similarly define $W_r^c(G_\dt)$ for closed walks. Note that these $\FI$-modules may not be well defined if the transition maps of $G_\dt$ are not injective. While injectivity may not be the case for small $n$, Theorem \ref{fgprops} implies that it certainly will be the case for $n \gg 0$. Therefore, we simply define $W_r(G_n)$, and $W_r^c(G_n)$ to be zero before injectivity takes effect. It is clear that $W_r^c(G_\dt)$ is a submodule of $W_r(G_\dt)$, and so to prove the corollary it will suffice to show that $W_r(G_\dt)$ is finitely generated. To prove this, we simply note that there is an embedding,
\[
W_r(G_\dt) \hookrightarrow \R V(G_\dt)^{\otimes r}
\]
defined on points by
\[
(v_1,\ldots,v_r) \mapsto v_1 \otimes \ldots \otimes v_r
\]
The module $\R V(G_\dt)^{\otimes r}$ is finitely generated in degree $\leq r \cdot d$ by Proposition \ref{tensorgen}. The Noetherian property concludes the proof.\\
\end{proof}

The above work illustrate how certain invariants of $G_n$ can grow with $n$. We also find, however, there are some invariants which must eventually stabilize.\\

\begin{corollary}\label{diastable}
Let $G_\dt$ denote a vertex-stable $\FI$-graph. Then the following invariants are independent of $n$ for $n \gg 0$:
\begin{enumerate}
\item the diameter of $G_n$;
\item the \textbf{girth} (i.e. the size of the smallest cycle) of $G_n$;
\end{enumerate}
\end{corollary}

\begin{proof}
For both statements, it suffices to show that the relevant invariant is eventually weakly decreasing in $n$. If $n \gg 0$, and $u,v \in V(G_n)$, then by Theorem \ref{fgprops} there exists $x,y \in V(G_{n-1})$ and an injection $f:[n-1] \hookrightarrow [n]$ such that $u = G(f)(x), v = G(f)(y)$. In particular, if $P$ is any path in $G_{n-1}$ connecting $x$ and $y$, then $f_\as(P)$ is a path in $G_n$ connecting $u$ to $v$. This shows that the shortest path between $u$ and $v$ cannot be longer than the shortest path between $x$ and $y$. By definition, the diameter of $G_n$ cannot be bigger than the diameter of $G_{n-1}$. A similar argument works for girth.\\
\end{proof}

\subsection{Topological consequences of vertex-stability}

In this section we consider a collection of topological applications of vertex-stability. Our first applications are simple consequences of the work in the previous section, as well as facts from our background sections.\\

\begin{lemma} \label{homcomplexfg}
Let $T$ be a graph, $G_\dt$ a vertex-stable $\FI$-graph of stable degree $\leq d$, and $\C^T_{n,i}$ denote the free $\Z$-module with basis indexed by the $i$-simplices of $\fiHom(T,G_n)$. Then $\C^T_{\dt,i}$ can be endowed with the structure of a finitely generated $\FI$-module over $\Z$ which is $((i+1)d\cdot |V(T)|)$-small.\\
\end{lemma}

\begin{proof}
We first recall the definition of the Hom-complex $\fiHom(T,G_n)$. The simplicies of $\fiHom(T,G_n)$ are multi-homomorphisms, and $\alpha$ is a face of $\tau$ if and only if $\alpha(x) \subseteq \tau(x)$ for all $x \in V(T)$. It is clear that the transition maps of $G_\dt$ induce the transition maps of $\C^T_{\dt,i}$, turning this collection of abelian groups into an $\FI$-module over $\Z$.\\

We have that $i$-simplices correspond to multi-homomorphisms with
\[
\sum_{x \in V(T)}|\alpha(x)| =|V(T)| + i
\]
The data of an $i$-simplex can therefore be encoded as an $|V(T)|$-tuple
\[
(\alpha_x)_{x \in V(T)}
\]
such that:

\begin{itemize}
\item $\alpha_x$ is a non-empty subset of $V(G_n)$;
\item $\sum_{x \in V(T)} |\alpha_x| = |V(T)|+i$;
\item if $\{x,y\} \in E(T)$ then for all $v \in \alpha_x$ and $w \in \alpha_y$, $\{v,w\} \in E(G_n)$.
\end{itemize}

Just as we have done previously, such as in the proof of Theorem \ref{fgprops}, we may realize $\C^T_{\dt,i}$ as a submodule of
\[
\left(\bigoplus_{j=1}^{i+1}\R\binom{V(G_\dt)}{j}\right)^{\otimes V(T)}
\]
The Noetherian property, as well as previous discovered facts about the modules $\binom{V(G_\dt)}{j}$ (see the proof of Theorem \ref{fgprops}) imply our lemma.\\
\end{proof}

Lemma \ref{homcomplexfg} is the main tool we will need in proving that Hom-complexes of vertex-stable $\FI$-graphs are representation stable in the sense of Church and Farb. Before we get to this theorem, we observe the following consequence of the above in terms of counting homomorphisms into $G_\dt$.\\

\begin{corollary}\label{chrompoly}
Let $T$ be any graph, and let $G_\dt$ be a vertex-stable $\FI$-graph of stable degree $\leq d$. Then for $n \gg 0$, then the function
\[
n \mapsto |\Hom(T,G_\dt)|
\]
agrees with a polynomial of degree $\leq d\cdot|V(T)|$.\\
\end{corollary}

\begin{proof}
This follows from Theorem \ref{repstab}, Lemma \ref{homcomplexfg}, as well as the fact that the module $\C^T_{n,0}$ from Lemma \ref{homcomplexfg} has basis indexed by $\Hom(T,G_n)$.\\
\end{proof}

\begin{remark}
It is easily seen that homomorphisms $\Hom(T,K_n)$ are in bijection with vertex colorings of $T$ for which no adjacent vertices are of the same color. The above corollary therefore recovers the existence of the so-called \textbf{chromatic polynomial}. Note that the chromatic polynomial exists for all $n \geq 0$, while the above only guarantees it for $n \gg 0$. One can recover the fact that the chromatic polynomial exists for $n \geq 0$ by showing that the collection of vector spaces $\R \Hom(T,K_\dt)$ can be endowed with the structure of an $\FI_\sharp$-module (see \cite{CEF}).\\

Note that a similar idea, i.e. using $\FI$-module techniques to recover the chromatic polynomial, was conveyed to the authors by John Wiltshire-Gordon and Jordan Ellenberg. This alternative technique was very similar in spirit, but used FA-modules instead of $\FI$-modules. Here, FA is the category of finite sets and all maps (see, for instance, \cite{WG}).\\
\end{remark}

\begin{theorem}\label{homcomrepstab}
Let $T$ be a graph, $G_\dt$ a vertex-stable $\FI$-graph of stable degree $\leq d$, and let $i \geq 0$ be a fixed integer. Then the $\FI$-module over $\Z$
\[
H_i(\fiHom(T,G_\dt))
\]
is $((i+1)d\cdot |V(T)|)$-small.\\
\end{theorem}

\begin{proof}
Recall the groups $\C^T_{n,i}$ from the Lemma \ref{homcomplexfg}. Standard simplicial homology informs us that there is a complex,
\[
\C^T_{n,\star}: \ldots \rightarrow \C^T_{n,i} \stackrel{\partial}{\rightarrow} \C^T_{n,i-1} \rightarrow \ldots \rightarrow \C^T_{n,0} \rightarrow 0
\]
with homology isomorphic to $H_\star(\fiHom(T,G_n))$. Lemma \ref{homcomplexfg} tells us that for each fixed $i$ the groups $\C^T_{\dt,i}$ form a finitely generated $\FI$-module over $\Z$. It isn't hard to show that the action of the transition maps of $\C^T_{\dt,i}$ commute with the differentials $\partial$. It follows that there is a complex of $\FI$-modules over $\Z$
\[
\C^T_{\dt,\star}: \ldots \rightarrow \C^T_{\dt,i} \stackrel{\partial}{\rightarrow} \C^T_{\dt,i-1} \rightarrow \ldots \rightarrow \C^T_{\dt,0} \rightarrow 0
\]
whose homology agrees with the $\FI$-modules $H_i(\fiHom(T,G_\dt))$. The Noetherian property and Lemma \ref{homcomplexfg} imply our result.\\
\end{proof}

Given an $\FI$-simplicial-complex $X_\dt$, work of Dochtermann \cite{Do2} implies that $X_\dt$ can actually be realized as a Hom-complex of some $\FI$-graph $G_\dt$. That work also constructs the graphs $G_n$ explicitly. The above theorem therefore reduces the problem of determining whether the homology groups of $X_\dt$ are representation stable to the combinatorial problem of determining whether an $\FI$-graph is vertex-stable. We will later see another simple condition with which one can use to determine whether the homology of an $\FI$-simplicial complex is representation stable (see Corollay \ref{repstabcom}).\\

To conclude this section we review some fundamental concepts and definitions which will be used in the proof of Theorem \ref{repstabconf}.\\

For the remainder of this section, we fix a vertex-stable, torsion-free $\FI$-graph $G_\dt$ as well as a positive integer $m$. We will assume that $G_\dt$ has stable degree $\leq d$ and edge-stable degree $\leq d_E$.\\

To begin, we note that the necessary edge subdivisions of Theorem \ref{cellularmodel} can be accomplished in a way consistent with the $\FI$-module structure of $G_\dt$.\\

\begin{proposition}\label{subdivfi}
There exists an $\FI$-graph, $G^{(m)}_\dt$, for which $G^{(m)}_n$ is the $m$-th subdivision of $G_n$ and for any injection of sets $f:[n] \hookrightarrow [r]$ one has
\[
G^{(m)}(f)(x) = G(f)(x)
\]
for all $x \in V(G_n)$. If $G_\dt$ has stable degree $\leq d$ and edge-stable degree $\leq d_E$, then $G^{(m)}_\dt$ has stable degree $\leq \max\{d,d_E\}$ and edge-stable degree $\leq d_E$.\\
\end{proposition}

\begin{proof}
The existence of $G^{(m)}$ from the fact that $G_\dt$ is torsion-free and the definition of the $m$-th subdivision. The statement on stable degrees follows from the fact that subdivision creates new vertices and edges within existing edges.\\
\end{proof}

This proposition will prove to be critical for us, as it essentially asserts, with Theorem \ref{cellularmodel}, that there exists a combinatorial model of $\Conf_m(G_\dt)$ which interacts nicely with the $\FI$-graph structure of $G_\dt$.\\

We are now ready to provide the main novel computational construction of this section. Recall that we have fixed a vertex-stable torsion-free $\FI$-graph $G_\dt$.\\

\begin{definition}
Fix integers $m,n,i \geq 0$. We write $\K_{n,m,i}$ to denote the free $\Z$-module with basis vectors indexed by the $i$-dimensional cells of the cubical complex $\DConf_m(G^{(m)}_n)$. Given any injection of sets $f:[n] \hookrightarrow [r]$, Proposition \ref{subdivfi} implies that the transition map $G(f)$ induces a transition map $G^{(m)}_n \rightarrow G^{(m)}_r$, which, in turn, induces a map
\[
f_\as: \K_{n,m,i} \rightarrow \K_{r,m,i}.
\]
This procedure equips the family $\{\K_{n,m,i}\}_n$ with the structure of an $\FI$-module over $\Z$.\\
\end{definition}

Having observed the $\FI$-module structure on the families $\K_{\dt,m,i}$, the strategy of our proof of Theorem \ref{repstabconf} becomes clear. We begin by proving that, for all choices of $m$ and $i$, the $\FI$-module $\K_{\dt,m,i}$ is finitely generated. In fact, we will prove that $\K_{\dt,m,i}$ is $(\max\{d,d_E\}(m-i) + d_Ei)$-small, where $d$ is the stable degree of $G_\dt$. Following this, one notes that the action of $\FI$ on the collection $\{K_{n,m,i}\}_n$ clearly commutes with the usual differentials
\[
\partial_{n,i,m}: \K_{n,m,i} \rightarrow \K_{n,m,i-1}.
\]
This implies that the collection of complexes
\[
\ldots \rightarrow \K_{n,m,i} \rightarrow \K_{n,m,i-1} \rightarrow \ldots \rightarrow \K_{n,m,0} \rightarrow 0
\]
can be pieced together to form a complex of $\FI$-modules. The Noetherian property is then sufficient for us to prove the main theorem.\\

We observe that this approach has the downside that it cannot be used to estimate the generating degree of the $\FI$-module over $\Z$, $H_i(\Conf_m(G_\dt))$. It is the belief of the authors that proving a result of this kind will require a deeper topological understanding of the spaces $\Conf_m(G_n)$ as $n$-varies. This seems like a rich avenue for future research, as surprisingly little is thus far understood about these spaces.\\

\begin{theorem}\label{strongconf}
Assume that $G_\dt$ has stable degree $\leq d$ and edge-stable degree $\leq d_E$. Then for all choices of $m,i \geq 0$, the $\FI$-module over $\Z$, $\K_{\dt,m,i}$, is $(\max\{d,d_E\}(m-i) + d_Ei)$-small.\\
\end{theorem}

\begin{proof}
We first define $\R E(G_\dt)$ to be the $\FI$-module of edges of $G_\dt$. It follows from definition that $\R V(G_\dt)$ is generated in degrees $\leq d$, while $\R E(G_\dt)$ is generated in degrees $\leq d_E$. Propositions \ref{tensorgen} and \ref{subdivfi} imply that the $\FI$-module
\[
Q_i := \bigoplus_{f:[n] \rightarrow \{V,E\}, |f^{-1}(E)| = i}\bigotimes_{j=1}^n Q_{f,j}
\]
is generated in degrees $\leq \max\{d,d_E\}(m-i) + d_Ei$, where,
\[
Q_{f,j} = \begin{cases} \R E(G^{(m)}_\dt) &\text{ if $f(j) = E$}\\ \R V(G^{(m)}_\dt) &\text{ otherwise.}\end{cases}
\]

It follows from definition that $\K_{\dt,m,i}$ is a submodule of $Q_i$, whence it is $(\max\{d,d_E\}(m-i) + d_Ei)$-small. This concludes the proof.\\
\end{proof}

\begin{remark}
In most cases of interest, $\max\{d,d_E\} = d_E$. Indeed, this will be the case so long as the graphs $G_n$ are all connected. The above theorem therefore implies that the homology groups $H_i(\Conf_m(G_n))$ have Betti numbers which agree with a polynomial of degree $\leq d_Em$ for $n \gg 0$ whenever $G_\dt$ is connected. Remark \ref{small} then implies this polynomial has degree $\leq 2dm$.\\
\end{remark}

\begin{example}
Let $G_\dt$ be the $\FI$-graph of Example \ref{lutex}. Then the above implies that the Betti numbers of $H_i(\Conf_m(G_n))$ eventually agree with a polynomial of degree $\leq m$. This bound is sharp for $m \geq 2, and i = 1$ in the case wherein $G$ is a single point, and $H$ is an edge. Namely, the case where $G_\dt = Star_\dt$ (see \cite{Gh} for this computation).\\
\end{example}

\subsection{Algebraic consequences of vertex-stability} \label{algap}

In this section, we consider adjacency and Laplacian matrices associated to an $\FI$-graph. We focus on properties of the eigenspaces associated to these matrices.\\

To begin, note that we may view the adjacency and Laplacian matrices of a graph $G$ as linear endomorphisms of $\R V(G)$. Given an $\FI$-graph $G_\dt$, it is unfortunately not the case that the collections $A_{G_\dt}$ and $L_{G_\dt}$ can be considered as endomorphisms of the $\FI$-module $\R V(G_\dt)$. Despite this, we will find that these matrices have some surprising interactions with the $\FI$-module structures. To begin, we have the following key observation.\\

\begin{lemma}\label{eigenrep}
Let $G_\dt$ be an $\FI$-graph. Then for each $n$ the matrices $A_{G_n}$ and $L_{G_n}$ commute with the action of $\Sn_n$. In particular, the eigenspaces of these matrices are sub-representations of $\R V(G_n)$.\\
\end{lemma}

\begin{proof}
For a fixed vertex $v \in V(G_n)$, we write $N(v)$ to denote the collection of vertices adjacent to $v$. Then,
\[
A_{G_n}v = \sum_{w \in N(v)} w.
\]
Therefore, if $\sigma \in \Sn_n$,
\[
A_{G_n}\sigma(v) = \sum_{w \in N(\sigma(v))} w = \sum_{\sigma(w'), w' \in N(\sigma(v))} \sigma(w') = \sigma(A_{G_n}v).
\]
The same proof works for the Laplacian matrix.\\

The second half of the lemma follows from basic linear algebra facts. Namely, if two matrices commute, then they preserve each others' eigenspaces.\\
\end{proof}

As an immediate consequence of Lemma \ref{eigenrep}, we obtain the following:

\begin{proposition}\label{boundedeval}
Let $G_\dt$ be a vertex-stable $\FI$-graph. Then there are constants $c_A$ and $c_L$, independent of $n$, such that the number of distinct eigenvalues of the adjacency matrix (resp. the Laplacian) of $G_n$ is bounded by $c_A$ (resp. $c_L$) for all $n$.\\
\end{proposition}

\begin{proof}
This follows from Lemma \ref{eigenrep} as well as the third part of Theorem \ref{repstab}.\\
\end{proof}

\begin{remark}
Proposition \ref{boundedeval} can be used in certain cases to prove that certain families of graphs \emph{cannot} be endowed with the structure of a vertex-stable $\FI$-graph. For example, the cycle graphs $C_n$ and the wheel graphs $W_n$ have $n$ distinct eigenvalues.\\
\end{remark}

In fact, Proposition \ref{boundedeval} is the first piece of evidence describing a much more robust structure. The following theorem follows as a consequence of upcoming work of David Speyer and the first author.\\

\begin{theorem}
Let $G_\dt$ denote a vertex-stable $\FI$-graphs. Then there exist constants $c_A,c_L$ such that for all $n \gg 0$, $A_{G_n}$ (resp. $L_{G_n}$) has $c_A$ (resp. $c_L$) distinct eigenvalues. For $i= 1,\ldots c_A$ (resp. $i = 1,\ldots, c_L$) and $n \gg 0$, let $\lambda^A_i(n)$ (resp. $\lambda^L_i(n)$) denote the $i$-th largest eigenvalue of $A_{G_n}$ (resp. $L_{G_n}$). Then the for all $i$ and all $n \gg 0$ the functions
\[
n \mapsto \text{ the multiplicity of }\lambda_i^A(n), \hspace{1cm} n \mapsto \text{ the multiplicity of }\lambda_i^L(n)
\]
each agree with a polynomial.\\
\end{theorem}

\begin{example}
We illustrate the above theorem with some examples. Let $G_\dt = K_\dt$ denote the $\FI$-graph of complete graphs. Then for $n \geq 1$, the $\Sn_n$-representation $\R V(G_n)$ is isomorphic to the usual permutation representation on $\R^n$. This decomposes into a pair of irreducible representations
\[
\R^n \cong \R \oplus S_n,
\]
where $\R$ is the trivial representation, and $S_n$ is the standard irreducible $(n-1)$-dimensional representation of $\Sn_n$. We note that the decomposition $\R^n \cong \R \oplus S_n$ agrees with the eigenspace decomposition of $\R V(G_n)$ with respect to both the adjacency matrix and the Laplacian matrix. The trivial representation is the eigenspace for $n-1$ (resp. 0), while $S_n$ is the eigenspace for -1 (resp. $-n$). It is easy to see that the collection $S_\dt$ actually forms a submodule of the $\FI$-module $\R V(G_n)$, and is therefore finitely generated. This implies that $\dim_\R S_n$ agrees with a polynomial for $n \gg 0$, which implies the same about the eigenvalue multiplicities in question.\\

Next, fix $m \geq 1$ and let $G_\dt =$ Star$_\dt = K_{\dt,1}$. For simplicity we only work with the eigenspaces for the adjacency matrix, although the Laplacian is not much different. For $n \geq 1$, the distinct eigenvalues of $G_n$ are $\pm \sqrt{n}$ and 0. We may decompose the representation $\R V(G_n)$ as
\[
\R V(G_n) = \R \oplus \R \oplus S_n
\]
where $S_n$ is as in the previous example, and $\R$ is once again the trivial representation. As before, this decomposition of $\R V(G_n)$ as a representation corresponds exactly to its decomposition in terms of eigenspaces.\\

Of course, one should not expect these eigenspaces to be irreducible as $\Sn_n$-representations in general. For example, if we instead consider $G_\dt = K_{\dt,m}$, where $m > 1$, then the eigenspaces of the adjacency matrix are not all irreducible as $\Sn_n$-representations. Despite this, one finds that the eigenspaces of 0 can be collected into a submodule of $\R V(G_\dt)$. Indeed, the proof of the previous theorem involves proving that these matrices can be made into morphisms of $\FI$-modules by restricting to their actions on the isotypic pieces of $\R V(G_n)$.\\
\end{example}

\section{Generalizations and alterations}
In this final section, we briefly discuss how the work of the previous sections can be generalized and altered to prove facts about different families of graphs and other simplicial complexes. We begin by considering graphs over categories other than $\FI$, and then move on to higher dimensional analogues to the previous work. Note that these sections are intended to be more motivation for further study, and should by no means be considered exhaustive.\\

\subsection{Other categories}

The representation theory of categories has seen a recent explosion in the literature, largely motivated by its connections with representation stability. In this section we consider representations of the categories $\VI(q)$, where $q$ is a power of a prime, and $\FI^m$, where $m$ is a positive integer. These categories can be seen discussed in \cite{Gad1,GW,LY,PS,SS}.\\

\begin{definition}\label{othercat}
Let $m$ be a fixed positive integer, and let $q$ be a power of a fixed prime $p$. The category $\VI(q)$ is that whose objects are free vector spaces over the finite field $\F_q$, and whose morphisms are injective linear maps. The category $\FI^m$ is defined to be the categorical product of $\FI$ with itself $m$ times. That is, it is the category whose objects are $m$-tuples of non-negative integers $(n_1,\ldots,n_m)$, and whose morphisms are $m$-tuples of injective maps $(f_1,\ldots,f_m):[n_1] \times \ldots \times [n_m] \hookrightarrow [n'_1,\ldots,n'_m]$.\\
\end{definition}

One may think of $\VI(q)$ as an analog of $\FI$, where the relevant acting groups are the finite general linear groups $GL(n,q)$. Similarly, $\FI^m$ is the analog of $\FI$ where the relevant acting groups are $\Sn_{n_1} \times \ldots \times \Sn_{n_m}$. Just as with $\FI$, a module over either of these categories will be defined to be a morphism from the category to $\R$ vector spaces. Definitions such as finite generation carry over in the obvious way\\

The following facts can be found in \cite{Gad1,GW,LY,PS,SS}.\\

\begin{theorem}
Let $\C$ denote either the category $\FI^m$ or $\VI(q)$. Then:
\begin{enumerate}
\item \cite{Gad1,LY} If $\C = \FI^m$, and $V$ is a finitely generated $\C$-module, then there exists a polynomial $p_V(x_1,\ldots,x_m) \in \Q[x_1,\ldots,x_m]$ such that for all $(n_1,\ldots,n_m)$ with $\sum_i n_i \gg 0$,
\[
\dim_\R(V_{n_1,\ldots,n_m}) = p_V(n_1,\ldots,n_m)
\]
\item \cite{SS,Gad1} If $V,W$ are finitely generated $\C$-modules, then the same is true of $V \otimes W$. \label{genten}
\item \cite{GW} If $V$ is a finitely generated $\VI(q)$-module, then there exists a polynomial $p_V(x) \in \Q[x]$ such that for all $n \gg 0$
\[
p_V(q^n) = \dim_\R V(\F_q^n).
\]
\item If $V$ is a finitely generated $\C$-module, then the transition maps of $V$ are all eventually injective.
\item \cite{PS,SS} If $V$ is a finitely generated $\C$-module, then all submodules of $V$ are also finitely generated.\\
\end{enumerate}
\end{theorem}

As one can see, these two categories have very similar properties to $\FI$-modules. Indeed, it is sufficient for us to recover virtually everything that was proven in previous sections.\\

\begin{definition}
Let $\C$ denote either the category $\FI^m$ or $\VI(q)$. Then a \textbf{$\C$-graph} is a functor $G_\dt:\C \rightarrow \mathbf{Graph}$. We say that $G_\dt$ is \textbf{vertex-stable} if the associated $\C$-module $\R V(G_\dt)$ is finitely generated.\\
\end{definition}

Borrowing notation and proofs from the previous sections, we conclude the following.\\

\begin{theorem} \label{bigthmgen}
Let $\C$ denote either $\VI(q)$ or $\FI^m$, and let $G_\dt$ be a vertex-stable $\C$-graph. Then:
\begin{enumerate}
\item the $\C$-module $\R E(G_\dt)$ is finitely generated;
\item for any $r \geq 1$, the $\C$-module $\binom{V(G_\dt)}{r}$ is finitely generated;
\item if $\C = \FI^m$, and $H$ is any fixed graph, then there exist polynomials $p_H(x_1,\ldots,x_m),p_H^{ind}(x_1,\ldots,x_m) \in \R[x]$ such that for all $\mathbf{n} := (n_1,\ldots,n_m)$ with $\sum_i n_i \gg 0$
\[
p_H(\mathbf{n}) = \eta_H(G_{\mathbf{n}}), \text{ and } p_H^{ind}(\mathbf{n}) = \eta_H^{ind}(G_{\mathbf{n}})
\]
\item if $\C = \VI(q)$, and $H$ is any fixed graph, then there exist polynomials $p_H(x),p_H^{ind}(x) \in \Q[x]$ such that for all $n \gg 0$
\[
p_H(q^n) = \eta_H(G_{\F_q^n}), \text{ and } p_H^{ind}(q^n) = \eta_H^{ind}(G_{\F_q^n})
\]
\item if $\C = \FI^m$, and $r \geq 1$ is fixed, then there exist polynomials $p_r(x),p_r^{c}(x) \in \Q[x_1,\ldots,x_m]$ such that for all $\mathbf{n} := (n_1,\ldots,n_m)$ with $\sum_i n_i \gg 0$
\[
p_r(\mathbf{n}) = |\{\text{number of walks in $G_{\mathbf{n}}$ of length $r$}\}|, \text{ and } p_r^{c}(\mathbf{n}) =|\{\text{number of closed walks in $G_{\mathbf{n}}$ of length $r$}\}|.
\]
\item if $\C = \VI(q)$, and $r\geq 1$ is fixed, then there exist polynomials $p_r(x),p_r^c(x) \in \Q[x]$ such that for all $n \gg 0$
\[
p_r(q^n) = |\{\text{number of walks in $G_n$ of length $r$}\}|, \text{ and } p_r^{c}(q^n) = |\{\text{number of closed walks in $G_n$ of length $r$}\}|
\]
\item for any fixed $i$, and any fixed graph $T$ the $\C$-module over $\Z$, $H_i(\Hom(T,G_\dt))$, is finitely generated;
\item if $G_\dt$ is torsion-free then for any fixed $m,i$ the $\C$-module over $\Z$, $H_i(\Conf_m(G_n))$, is finitely generated.\\
\end{enumerate}
\end{theorem}

To conclude this section, we consider various natural examples of $\FI^m$ and $\VI(q)$ graphs. The reader should keep in mind Theorem \ref{bigthmgen} while reading what follows.\\

\begin{example}
Recall that for fixed $m$, we considered the vertex-stable $\FI$-graph $K_{\dt,m}$. While this yielded various results, it is perhaps more correct to allow $m$ to vary, and consider the vertex-stable $\FI^2$-graph $K_{\dt_1,\dt_2}$. More generally, we can consider the complete $r$-partite graph $K_{\dt_1,\ldots,\dt_r}$ as a vertex-stable $\FI^r$-graph.\\

If $G,H$ are any graphs, then there are multiple ways one can define the product of $G$ and $H$. One such method is with the tensor (or categorical) product $G \times H$. The graph $G \times H$ is that whose vertex set is given by $V(G \times H) = V(G) \times V(H)$ and for which $\{(x_1,y_1),(x_2,y_2)\} \in E(G \times H)$ if and only if $\{x_1,x_2\},\{y_1,y_2\} \in E(G) \cup E(H)$. If $G_\dt$ and $H_\dt$ are two vertex-stable $\FI$-graphs, then we may define the $\FI^2$-graph $G_\dt \times H_\dt$ by the assignments
\[
(G_\dt \times H_\dt)_{n_1,n_2} = G_{n_1} \times H_{n_2}.
\]
It is clear that this family is vertex-stable as an $\FI^2$-graph. Note that a similar statement will hold for many of the other common graph products such as strong products and carteasean products (see any standard reference on algebraic graph theory for definitions of these products such as \cite{B}).\\

Turning our attention to $\VI(q)$, one is immediately reminded of the Grassmann graphs $J_q(n,k)$. The vertices of $J_q(n,k)$ are $k$-dimensional subspaces of $\F_q^n$, and two vertices form an edge if and only if the intersection of the corresponding subspaces is non-empty. Note that one may think of $J_q(n,k)$ is a ``$q$-version'' of the Johnson graph $J(n,k)$. In fact, many of the $\FI$-graphs we previously studied will have associated $\VI(q)$-graphs. For instance, we may define $KG_q(n,k)$ to be the graph whose vertices are subspaces of $\F_q^n$ of dimension $k$, and for which two vertices are connected if and only if their corresponding subspaces are disjoint.\\
\end{example}

\subsection{$\FI$-simplicial-complexes}

In this section, we generalize the work of the previous sections to higher dimensional simplicial complexes.\\
 
\begin{definition}
Let $X$ be a (compact) simplicial complex. We will write $V_i(X)$ for the set of $i$-simplices of $X$. A \textbf{simplicial map} between simplicial complexes $X,Y$ is a continuous morphism $f:X\rightarrow Y$ such that $f(V_i(X)) \subseteq V_i(Y)$.\\

An \textbf{$\FI$-simplicial-complex} is a (covariant) functor from $\FI$ to the category of simplicial complexes and simplicial maps. Given an $\FI$-simplicial complex $X_\dt$, we write $\R V_0(X_\dt)$ for the $\FI$-module whose evaluation at $[n]$ is the vector space with basis indexed by $\R V_0(X_n)$. We similarly define the $\FI$-modules $\R V_i(X_\dt)$, and $\binom{V_i(X_\dt)}{r}$ for all $i,r \geq 0$.\\

We say that an $\FI$-simplicial complex $X_\dt$ is \textbf{vertex stable with stable degree $\leq d$} if the $\FI$-module $\R V_0(X_\dt)$ is finitely generated in degree $\leq d$.\\
\end{definition}

The following theorem is proven in the exact same way as Theorem \ref{fgprops}.\\

\begin{theorem} \label{highdimfgprops}
Let $X_\dt$ be a vertex-stable $\FI$-simplicial-complex with stable degree $\leq d$. Then:
\begin{enumerate}
\item For all $i$ the $\FI$-modules $\R V_i(X_\dt)$ are $(d(i+1))$-small;
\item For all $i,r$ the $\FI$-modules $\binom{V_i(X_\dt)}{r}$ are $(rd(i+1))$-small;
\end{enumerate}
\end{theorem}

From the perspective of representation stability, the above reveals something a bit striking about $\FI$-simplicial-complexes. Note that the following was also true about graphs, though in that case it is less interesting.\\

\begin{corollary}\label{repstabcom}
Let $X_\dt$ be a vertex-stable $\FI$-simplicial-complex with stable degree $\leq d$. Then for all $i \geq 0$ the $\FI$-module over $\Z$
\[
H_i(X_\dt)
\]
is $di$-small.\\
\end{corollary}

\begin{proof}
Follows from the first part of Theorem \ref{highdimfgprops}, as well as the usual complex for computing simplicial homology and the Noetherian property of $\FI$-modules.\\
\end{proof}

Note that we could have used the above corollary to shorten the proof of Theorem \ref{homcomrepstab}.\\

Another interesting corollary of Theorem \ref{highdimfgprops} relates to counting colorings of a simplicial complex $T$. Just as in the graph case, the colorings we will consider are intimately linked with simplicial maps into a certain $\FI$-simplicial-complex. To begin, we therefore note the following.\\

\begin{theorem} \label{homfg}
Let $X_\dt$ denote a vertex-stable $\FI$-graph. Then for any simplicial complex $T$, the $\FI$-module,
\[
\R \Hom(T,X_\dt),
\]
whose evaluation on $[n]$ is the real vector space with basis indexed by $\Hom(T,X_n)$, is finitely generated. In particular, for $n \gg 0$, the function
\[
n \mapsto |\Hom(T,X_n)|
\]
agrees with a polynomial.\\
\end{theorem}

\begin{proof}
The $\FI$-module $\R \Hom(T,X_\dt)$ can be realized as a submodule of $(\R V_0(X_\dt))^{\otimes|V_0(T)|}$. This module is finitely generated by Proposition \ref{tensorgen}, so the Noetherian property implies our result.\\
\end{proof}

\begin{definition}
Let $(r,s)$ be a pair of positive integers. An \textbf{$(r,s)$-coloring} of a simplicial complex $T$ is a map of sets $f:V_0(T) \rightarrow [r]$ such that if $\{v_1,\ldots,v_i\} \in V_i(T)$, then at most $s$ of the vertices $v_1,\ldots,v_i$ share the same color.\\

If $(r,s)$ is as above, we define the \textbf{$(r,s)$ complete simplicial complex}, $K^s_r$, to be the simplicial complex whose vertices are given by,
\[
V_0(K_r^s) := [r] \times [s]
\]
and whose $i$-simplices are all $i$-element subsets of $V_0(K_r^s)$. If $f:[r] \hookrightarrow [r']$ is an injection, then we obtain a simplicial map $f_\as:K_r^s \rightarrow K_{r'}^s$. This turns the collection of $(r,s)$ complete simplicial complexes, with $s$ fixed, into a vertex-stable $\FI$-simplicial-complex
\[
K_\dt^s
\]
\end{definition}

Colorings of simplicial complexes have recently seen interest in the literature, and seem to have deep connections with Stanley-Reisner theory \cite{BFHT,DMN}. While one would like to claim that $(r,s)$-colorings of a simplicial complex $T$ are in bijection with simplicial maps into $K_r^s$, this is not quite the case. However, it is close enough that there are still conclusions that we may draw. \\

\begin{corollary}
Let $T$ be a simplicial complex, and let $s \geq 1$ be an integer. If we write $V^{T,s}_\dt$ to denote the $\FI$-module whose evaluation on $[r]$ is the real vector space with basis indexed by $(r,s)$-colorings of $T$, then $V^{T,s}_\dt$ is a quotient of the $\FI$-module
\[
\R \Hom(T,K_{\dt}^s),
\]
as defined in Theorem \ref{homfg}. In particular, for $r \gg 0$, the function
\[
r \mapsto |\{\text{$(r,s)$-colorings of T}\}|
\]
agrees with a polynomial.\\
\end{corollary}

\begin{proof}
Given an element $f \in \Hom(T,K^s_r)$, we may associate an $(r,s)$-coloring of $T$ by assigning $v \in V_0(T)$ the first coordinate of $f(v)$. This defines a surjective map of $\FI$-modules
\[
\R \Hom(T,K_{\dt}^s) \rightarrow V^{T,s}_\dt
\]
as desired.\\
\end{proof}

\end{document}